\documentclass{amsart}
\usepackage{amsmath,amsfonts,amsbsy,amsgen,amscd,mathrsfs,amssymb,amsthm}
\usepackage{mathtools}
\usepackage{hyperref}
\usepackage{enumerate}

\usepackage{pgfplots}

\usepackage{multido}

\numberwithin{equation}{section}
\newtheorem{theorem}{Theorem}[section]

\newtheorem{corollary}[theorem]{Corollary}
\newtheorem{lemma}[theorem]{Lemma}
\newtheorem{problem}[theorem]{Open Problem}
\newtheorem{proposition}[theorem]{Proposition}

\theoremstyle{definition}
\newtheorem{assumption}[theorem]{Assumption}
\newtheorem{definition}[theorem]{Definition}
\newtheorem{heuristic}[theorem]{Heuristic computation}
\newtheorem{remark}[theorem]{Remark}

\makeatletter\renewenvironment{proof}[1][\proofname] {\par\pushQED{\qed}\normalfont\topsep6\p@\@plus6\p@\relax\trivlist\item[\hskip\labelsep\bfseries#1\@addpunct{.}]\ignorespaces}{\popQED\endtrivlist}

\def\1{\mathbf 1}
\def\a{\mathbf a}
\def\al{\alpha}

\def\be{\beta}
\def\C{\mathcal{C}}
\def\dd{\mathrm{d}}

\def\E{\mathbf E}
\def\eps{\varepsilon}
\def\deq{\stackrel{\mathrm d}{=}}
\def\erf{\mathrm{erf}}
\def\erfc{\mathrm{erfc}}

\def\H{\mathcal{H}}

\def\im{\mathrm i}
\def\La{\Lambda}
\def\la{\lambda}
\def\M{\mathcal M}
\def\N{\mathbb N}
\def\R{\mathbb{R}}

\def\Sc{\mathrm{Sc}}
\def\sgn{\mathrm{sgn}}
\def\Tr{\mathrm{Tr}}
\def\U{\mathcal{U}}
\def\ubar{\underline}
\def\X{\boldsymbol X}
\def\Z{\mathbb{Z}}

\newcommand{\ep}{\hfill \ensuremath{\Box}}

\def\multiset#1#2{\ensuremath{\left(\kern-.3em\left(\genfrac{}{}{0pt}{}{#1}{#2}\right)\kern-.3em\right)}}

\renewcommand\bar{\overline}
\renewcommand\d{~\textnormal{d}}
\renewcommand\hat{\widehat}

\renewcommand\phi{\varphi}
\renewcommand\Pr{\mathbf{P}}

\begin{document}
\title[Spiked beta ensembles]{Edge of spiked beta ensembles, stochastic Airy semigroups and reflected Brownian motions} 
\author{Pierre Yves Gaudreau Lamarre}
\address{ORFE Department, Princeton University, Princeton, NJ 
08544, USA}
\email{plamarre@princeton.edu}
\author{Mykhaylo Shkolnikov}
\address{ORFE Department, Princeton University, Princeton, NJ 
08544, USA}
\email{mshkolni@gmail.com}
\thanks{P.~Y.~Gaudreau Lamarre is partially supported by an NSERC Doctoral Fellowship and a Gordon Y.~S.~Wu Fellowship. M.~Shkolnikov is partially supported by the NSF grant DMS-1506290.}

\begin{abstract}
We access the edge of Gaussian beta ensembles with one spike by analyzing high powers of the associated tridiagonal matrix models. In the classical cases $\beta=1,\,2,\,4$, this corresponds to studying the fluctuations of the largest eigenvalues of additive rank one perturbations of the GOE/GUE/GSE random matrices. In the infinite-dimensional limit, we arrive at a one-parameter family of random Feynman-Kac type semigroups, which features the stochastic Airy semigroup of Gorin and Shkolnikov \cite{GS} as an extreme case. Our analysis also provides Feynman-Kac formulas for the spiked stochastic Airy operators, introduced by Bloemendal and Vir\'{a}g \cite{BV1}. The Feynman-Kac formulas involve functionals of a reflected Brownian motion and its local times, thus, allowing to study the limiting operators by tools of stochastic analysis. We derive a first result in this direction by obtaining a new distributional identity for a reflected Brownian bridge conditioned on its local time at zero. 
\end{abstract}


\maketitle

%

\section{Introduction}

A remarkable advance in the study of random matrices and related point processes has been the development of a theory of operator limits for such objects in \cite{BV1}, \cite{BV2}, \cite{ES}, \cite{GS}, \cite{HF}, \cite{KRV}, \cite{RR}, \cite{RRV}, \cite{RW}, \cite{VV1}, \cite{VV2}. This line of research originates from the publications \cite{E}, \cite{ES} by Edelman and Sutton, who have realized that the random tridiagonal matrices of Dumitriu and Edelman \cite{DE} (see \eqref{Equation: Tridiagonal} below for a definition) can be viewed as finite-dimensional approximations of suitable random Schr\"odinger operators. Since the joint eigenvalue distributions in the Dumitriu-Edelman models for the parameter values $\beta=1,\,2,\,4$ are given by the eigenvalue point processes of the Gaussian orthogonal/unitary/symplectic ensembles (GOE/GUE/GSE), respectively, the insights of \cite{E}, \cite{ES} suggest that the limiting fluctuations of the largest eigenvalues of the latter can be read off from the random Schr\"odinger operators associated with the former. This approach has been carried out rigorously in the seminal paper \cite{RRV} by Ram\'{i}rez, Rider and Vir\'{a}g. We also refer to \cite{KRV} for a corresponding universality result, to \cite{BV1}, \cite{BV2} for extensions to spiked random matrix ensembles, to \cite{RR}, \cite{HF}, \cite{RW} for operator limits describing the fluctuations of the smallest eigenvalues of large positive definite random matrices, and to \cite{VV1}, \cite{VV2} for operators arising in the study of the bulk eigenvalues of random matrices. 

\medskip

More recently, Gorin and Shkolnikov \cite{GS} have proposed a different operator limit approach to the study of the largest eigenvalues in the Gaussian beta ensembles. The latter are point processes on the real line, in which the joint density of the points $\lambda_1\ge\lambda_2\ge\cdots\ge\lambda_N$ is proportional to 
\begin{equation}\label{Equation: Gibbs}
\prod_{1\le q_1<q_2\le N} (x_{q_1}-x_{q_2})^\be\;\prod_{q=1}^N e^{-\be x_q^2/4}.
\end{equation}
For $\beta=1,\,2,\,4$, the Gaussian beta ensemble describes the eigenvalue process of a random matrix from the GOE/GUE/GSE, respectively (see e.g. \cite[Section 2.5]{AGZ}). Gaussian beta ensembles with general values of $\beta>0$ appear frequently in the statistical physics literature and are commonly known therein as ``log-gases'', see e.g. \cite[Section 4.1]{F}. 

\medskip

The starting point of \cite{GS} is the celebrated result of Dumitriu and Edelman \cite{DE} establishing \eqref{Equation: Gibbs} as the joint eigenvalue distribution, for \textit{all} values of $\beta>0$, of the random matrix
\begin{equation}\label{Equation: Tridiagonal}
H^{\be}_N:=\frac{1}{\sqrt{\be}}\left[
\begin{array}{ccccc}
\sqrt{2}G_1&\chi_{(N-1)\be}\\
\chi_{(N-1)\be}&\sqrt{2}G_2&\chi_{(N-2)\be}\\
&\chi_{(N-2)\be}&\sqrt{2}G_3&\ddots&\\
&&\ddots&\ddots&\chi_{\be}\\
&&&\chi_{\be}&\sqrt{2}G_N
\end{array}
\right],
\end{equation}
where $G_1,\,G_2,\,\ldots,\,G_N$ are independent standard Gaussian random variables, $\chi_{\beta}$, $\chi_{2\beta},\,\ldots,\,\chi_{(N-1)\beta}$ are independent chi random variables indexed by their parameters, and the chi random variables are independent of the Gaussian random variables. Then, the convergence in finite-dimensional distribution sense, as $N\to\infty$, of the fluctuations 
\begin{equation}\label{eq:edge_scaling}
\Lambda_{q,N}:=N^{1/6}\big(2\sqrt{N}-\lambda_q\big),\quad q=1,\,2,\,\ldots
\end{equation}  
to the eigenvalues $\Lambda_1\le\Lambda_2\le\cdots$ of the stochastic Airy operator 
\begin{equation}
\H^\be f:=\bigg(-\frac{\dd^2}{\dd x^2}+x+\frac{2}{\sqrt{\be}}W'_x\bigg) f,\quad f\in L^2([0,\infty)),\;\;f(0)=0
\end{equation}
(see \cite[Theorem 1.1]{RRV}), where $W'$ is the white noise on $[0,\infty)$, and the simple computation
\begin{equation}\label{eq:N2/3heuristic}
\bigg(\frac{\la_q}{2\sqrt{N}}\bigg)^{\lfloor TN^{2/3}\rfloor}
=\bigg(1-\frac{\Lambda_{q,N}}{2N^{2/3}}\bigg)^{\lfloor TN^{2/3}\rfloor}\to e^{-T\Lambda_q/2},\quad T\ge 0
\end{equation}
suggest the convergence of the random matrices $\Big(\frac{H_N^\beta}{2\sqrt{N}}\Big)^{\lfloor TN^{2/3}\rfloor}$, $N\in\N$ to $e^{-T\H^\be/2}$ in a suitable operator topology. The main result of \cite{GS} (see \cite[Theorem 2.8]{GS}) establishes a more general version of such an operator convergence directly, without relying on the findings of \cite{RRV}. In contrast to $\H^\be$, the operator $e^{-T\H^\be/2}$ is an integral operator, and the corresponding integral kernel can be written in terms of the Brownian motion $W$ and an independent Brownian bridge (see \cite[equation (2.4)]{GS}). This allows to study the properties of $e^{-T\H^\be/2}$ by tools of stochastic analysis (see \cite[Proposition 2.14]{GS}, \cite[Theorem 1.1]{H} for an example). 

\medskip

In this paper, we continue the program initiated in \cite{GS} and consider the case of Gaussian beta ensembles with one spike. To this end, it is useful to recall that the stochastic Airy operator $\H^\beta$ describes the limiting behavior of the largest eigenvalues in the Laguerre beta ensemble (see \cite{RRV}, \cite{DE}). The latter interpolates between the eigenvalue processes of sample covariance matrices $\X\X^*$, where the entries of $\X$ are independent standard Gaussian random variables. From the point of view of statistical applications, the Laguerre beta ensemble is arguably not the most interesting model, since the entries in the columns of $\X$ are uncorrelated. Instead, one often considers multiplicative perturbations of $\X\X^*$ of the form $\X\Sigma\X^*$,
where $\Sigma=\widetilde\Sigma_r\oplus I_{N-r}$ is the direct sum of a deterministic full rank $r\times r$ matrix $\widetilde\Sigma_r$ and the $(N-r)\times(N-r)$ identity matrix $I_{N-r}$. Such models are known in the literature as the spiked covariance models, and we refer to the introduction in \cite{BV1} for an excellent summary. 

\medskip

As first discovered by Baik, Ben Arous and P\'ech\'e \cite{BBP} in the case of complex covariance matrices, the fluctuations of the largest eigenvalues exhibit a phase transition (known as the BBP phase transition) depending on the size of the perturbation. In the subcritical regime, the perturbation $\Sigma$ is so insignificant that the limiting behaviour is the same as in the unperturbed case; in the critical regime, the fluctuation exponents are the same as in the unperturbed case, but the limiting distributions are different;
and in the supercritical regime, the size of the perturbation is so large that the largest eigenvalues of $\X\Sigma\X^*$ separate from the bulk of the spectrum.

\medskip

For rank one perturbations, the critical regime of the BBP phase transition has been analyzed in detail by Bloemendal and Vir\'ag \cite{BV1}, and we describe their main result in the case of an additive perturbation. The corresponding tridiagonal model 
\begin{equation}\label{Equation: Spiked Tridiagonal 1}
H^{\be;w}_N:=\frac{1}{\sqrt{\be}}\left[
\begin{array}{ccccc}
\sqrt{2}G_1+\sqrt{\be N}\ell_N&\chi_{(N-1)\be}\\
\chi_{(N-1)\be}&\sqrt{2}G_2&\chi_{(N-2)\be}\\
&\chi_{(N-2)\be}&\sqrt{2}G_3&\ddots&\\
&&\ddots&\ddots&\chi_{\be}\\
&&&\chi_{\be}&\sqrt{2}G_N
\end{array}
\right]
\end{equation}
can be obtained for $\beta=1,\,2,\,4$ by applying the Dumitriu-Edelman tridiagonalization procedure to the sum of a GOE/GUE/GSE matrix and a rank one matrix with non-zero eigenvalue $\sqrt{N}\ell_N$, where 
\begin{equation}\label{Equation: Critical Phase Scaling}
\lim_{N\to\infty}N^{1/3}(1-\ell_N)=w\in\R. 
\end{equation}
Then, for all $\beta>0$ and under the scaling of \eqref{eq:edge_scaling}, the ordered eigenvalues of $H^{\be;w}_N$ converge in finite-dimensional distribution sense, as $N\to\infty$, to the ordered eigenvalues of the spiked stochastic Airy operator
\begin{equation}\label{eq:SSAE}
\H^{\be;w}f:=\bigg(-\frac{\dd^2}{\dd x^2}+x+\frac{2}{\sqrt{\be}}W'_x\bigg)f,\quad f\in L^2([0,\infty)),\;\;f'(0)=wf(0).
\end{equation}
The case $w=\infty$ formally corresponds to the Dirichlet boundary condition $f(0)=0$, motivating the convention $\H^{\be;\infty}:=\H^\be$.

\begin{remark}
The limit in \eqref{Equation: Critical Phase Scaling} determines the regime in the BBP phase transition: a limit of $\infty$ corresponds to the subcritical regime, a finite limit to the critical regime, and a limit of $-\infty$ to the supercritical regime. 
\end{remark}

\medskip

We turn to our main results. For the sake of convenience, we work with a modification of the tridiagonal model \eqref{Equation: Spiked Tridiagonal 1}:   
\begin{equation}\label{Equation: Spiked Matrix Model}
M^{\be;w}_N:=\left[
\begin{array}{ccccc}
\sqrt{N}\ell_N&\sqrt{N}+\xi_0\\
\sqrt{N}+\xi_0&\a_1&\sqrt{N-1}+\xi_1\\
&\sqrt{N-1}+\xi_1&\a_2&\ddots&\\
&&\ddots&\ddots&1+\xi_{N-1}\\
&&&1+\xi_{N-1}&\a_N
\end{array}
\right].
\end{equation}
The following assumption summarizes the conditions we impose on the matrix entries throughout the paper. We emphasize that we allow the random variables $\a_1,\,\a_2,\,\ldots$ and $\xi_0,\,\xi_1,\,\ldots$ to vary with $N$, even though the dependence on $N$ is suppressed to simplify the notation. 

\begin{assumption}\label{Assumption: Matrix Assumption}
The random variables $\a_1,\,\a_2,\,\ldots$ and $\xi_0,\,\xi_1,\,\ldots$ are mutually independent and such that: 
\begin{enumerate}[(a)]
\item $\big|\E[\a_m]\big|=o\big((N-m)^{-1/3}\big)$ and $\big|\E[\xi_m]\big|=o\big((N-m)^{-1/3}\big)$ as $(N-m)\to\infty$, 
\item $\E[\a_m^2]=s_{\a}^2+o(1)$ and $\E[\xi_m^2]=s_\xi^2+o(1)$ as $(N-m)\to\infty$, where $s_{\a}$, $s_{\xi}$ are non-negative constants satisfying $\frac{s_{\a}^2}{4}+s_\xi^2=\frac{1}{\beta}$ for some $\beta>0$, 
\item $\E\big[|\a_m|^p\big]\le C^p p^{\gamma p}$ and $\E\big[|\xi_m|^p\big]\le C^p p^{\gamma p}$ for all $N$, $m$ and $p$, with some constants $C<\infty$ and $0<\gamma<2/3$.
\end{enumerate}
Moreover, we assume that the non-random sequence $\ell_N$, $N\in\N$ satisfies \eqref{Equation: Critical Phase Scaling} and use the convention $\a_0:=0$.
\end{assumption}

\begin{remark}\label{rmk:special entries}
Assumption \ref{Assumption: Matrix Assumption} holds, in particular, when the $\a_m$'s are chosen to be i.i.d. Gaussian with mean $0$ and variance $2/\be$, whereas the $\xi_m$'s are drawn independently such that each $\sqrt{\be}(\sqrt{N-m}+\xi_m)$ is a chi random variable with parameter $\be(N-m)$ (see \cite[Lemma 2.2]{GS}).
\end{remark}

Motivated by the computation in \eqref{eq:N2/3heuristic}, we consider the powers 
\begin{equation}\label{eq:script_M}
\M^{\be;w}_{T;N}:=\left(\frac{M^{\be;w}_N}{2\sqrt{N}}\right)^{\lfloor TN^{2/3}\rfloor},\quad T\geq0.
\end{equation}
The operator limits of the latter turn out to be given by the following definition. 

\begin{definition}\label{def:SAS}
For every $\be,T>0$ consider the operator
\begin{equation}\label{eq:U}
(\U^{\be;w}_Tf)(x):=\E_{R^x}\bigg[\exp\bigg(-\int_0^T\frac{R^x_t}2\d t+\int_0^\infty\frac{L_T^a(R^x)}{\sqrt{\be}}\d W_a-w\frac{L_T^0(R^x)}2\bigg)f(R^x_T)\bigg]
\end{equation}
acting on the space
\begin{equation}\label{eq:D}
{\mathcal D}:=\Big\{f\in L^1_{\mathrm{loc}}([0,\infty)):\;\big|f(x)\big|\le C_1 e^{C_2 x^{1-\delta}}\text{ for some }C_1,C_2<\infty,\,\delta\in(0,1)\Big\},
\end{equation}
where
\begin{enumerate}[(a)]
\item $R^x$ is a reflected Brownian motion started at $x\ge 0$,
\item $\E_{R^x}[\,\cdot\,]$ is the expectation with respect to $R^x$, 
\item the local time of $R^x$ is defined as the continuous version of
\begin{equation}
L_T^a(R^x):=\lim_{\eps\downarrow0}\frac1{\eps}\int_0^T\1_{[a,a+\eps)}(R^x_t)\d t,\quad a\ge0,
\end{equation}
\item $W$ is a standard Brownian motion independent of $R^x$,
\item the It\^o integral with respect to $W$ is defined pathwise, as per \cite{Ka}.
\end{enumerate}
\end{definition}

\begin{remark}\label{rmk:bridge_kernel}
A trivial restatement of Definition \ref{def:SAS} is that $\U^{\be;w}_T$ is a random integral operator with the kernel
\begin{equation}\label{Equation: Spiked Kernel}
\begin{split}
& K^{\be;w}_T(x,y):=\frac{\exp\left(-\frac{(x-y)^2}{2T}\right)+\exp\left(-\frac{(x+y)^2}{2T}\right)}{\sqrt{2\pi T}}\\
& \qquad\qquad\;\;\cdot\E_{R^x}\bigg[\exp\bigg(-\int_0^T\frac{R^x_t}2\d t+\int_0^\infty\frac{L_T^a(R^x)}{\sqrt{\be}}\d W_a-w\frac{L_T^0(R^x)}2\bigg)\bigg|R^x_T=y\bigg].
\end{split}
\end{equation}
\end{remark}

\begin{remark}\label{rmk:matrix as operator}
For each $N$, the matrix $\M^{\be;w}_{T;N}$ can be regarded as an integral operator acting on $L^1_{\mathrm{loc}}([0,\infty))$ by associating $\R^{N+1}$ with the subspace of step functions
\begin{equation}\label{Equation: Step Functions}
L^1_N([0,\infty)):=\bigg\{\sum_{l=0}^N v_l\1_{[N^{-1/3}l,N^{-1/3}(l+1))}:\;v_0,\,v_1,\,\ldots,\,v_N\in\R\bigg\},
\end{equation}
and then mapping functions $f\in L^1_{\mathrm{loc}}([0,\infty))$ into $L^1_N([0,\infty))$ via
\begin{equation}\label{Equation: Projection}
(\pi_N f)(x):=
\displaystyle\sum_{l=0}^N N^{1/6}\int_{N^{-1/3}l}^{N^{-1/3}(l+1)}f(y)\d y\cdot \1_{[N^{-1/3}l,N^{-1/3}(l+1))}(x)
\end{equation}
before acting with $\M^{\be;w}_{T;N}$ on them. 
\end{remark}

Our main convergence result reads as follows.

\begin{theorem}\label{Theorem: Main}
For every $\be>0$, $w\in\R$, and with ${\mathcal D}$ defined in \eqref{eq:D}, one has
\begin{equation}\label{eq:main theorem eq}
\forall\,f,g\in{\mathcal D},\;T\ge0:\quad
\lim_{N\to\infty}\big(\pi_Nf\big)^\top\M^{\be;w}_{T;N}\big(\pi_Ng\big)=
\int_0^\infty f(x)\,(\U^{\be;w}_Tg)(x)\,\mathrm{d}x,
\end{equation}
where the convergence is in distribution and in the sense of moments. Moreover, these convergences hold jointly for any finite collection of $T$'s, $f$'s and $g$'s, and in the case of the convergence in distribution also jointly with the convergence in distribution 
\begin{equation}\label{Equation: Origin of Noise}
\sqrt{\be}\lim_{N\to\infty}N^{-1/6}\sum_{m=0}^{\lfloor N^{1/3}x\rfloor}\left(\frac{\a_m}{2}+\xi_m\right)=W_x,\quad x\geq0
\end{equation}
with respect to the Skorokhod topology. Here $W$ is the Brownian motion from \eqref{eq:U}.
\end{theorem}

Next, we present some natural properties of the operators $\U^{\be;w}_T$, $T\geq0$, viewed as operators on $L^2([0,\infty))$, and, in particular, connect them to the spiked stochastic Airy operator $\H^{\be;w}$ in \eqref{eq:SSAE}. 

\begin{proposition}\label{Proposition: Properties}
For every $\be>0$ and $w\in\R$, the following statements hold.
\begin{enumerate}[(a)]
\item If the same Brownian motion $W$ is used in the definitions of $\H^{\be;w}$ and $\U^{\be;w}_T$, then for every $T\ge0$,
\begin{equation}\label{eq:relation to SAO}
\U^{\be;w}_T=\exp\bigg(-\frac{T}{2}\H^{\be;w}\bigg)\;\;\text{ almost surely},
\end{equation}
in the sense that if $(f_q,\La_q:q\in\N)$ are the eigenfunction-eigenvalue pairs of $\H^{\be;w}$, then $\U^{\be;w}_T$ is the unique operator on $L^2([0,\infty))$ with eigenfunction-eigenvalue pairs $(f_q,e^{-T\La_q/2}:\,q\in\N\big)$ almost surely.
\item The family $(\U^{\be;w}_T:\,T\geq0)$ has the almost sure semigroup property
in the sense that for all $T_1,T_2\geq0$,
one has $\U^{\be;w}_{T_1}\U^{\be;w}_{T_2}=\U^{\be;w}_{T_1+T_2}$ almost surely.
\item For every $T>0$, the operator $\U^{\be;w}_T$ is symmetric,
non-negative and belongs to the Hilbert-Schmidt class almost surely. 
\item For every $T>0$, the operator $\U^{\be;w}_T$ is almost surely trace class and obeys the trace formula
\begin{equation}\label{Equation: Spiked Trace Formula}
\Tr\big(\U^{\be;w}_T\big)=\int_0^\infty K^{\be;w}_T(x,x)\d x.
\end{equation}
\item The family $(\U^{\be;w}_T:\,T\geq0)$ is $L^2$-strongly continuous in expectation, that is, for all $p>0$, $T\geq0$ and $f\in L^2([0,\infty))$, one has
\begin{equation}\label{Equation: L2 strong}
\lim_{t\to T}\E\Big[\big\|\U^{\be;w}_Tf-\U^{\be;w}_tf\big\|^p_{L^2([0,\infty))}\Big]=0.
\end{equation}
\end{enumerate}
\end{proposition}

\begin{remark}
Proposition \ref{Proposition: Properties}(a) should be viewed as a Feynman-Kac formula for the spiked stochastic Airy operator $\H^{\be;w}$. 
\end{remark}

Remark \ref{rmk:bridge_kernel} shows that one might be able to understand observables of the limiting operators $\U^{\be;w}_T$, $T \ge0$, such as moments of certain linear statistics of their spectra, by investigating the corresponding functionals of reflected Brownian motions conditioned on their endpoints. As a first step in this direction, we consider $\E\big[K^{\be;w}_T(0,0)\big]$, which, in view of the next proposition, seems to be the simplest object to study. 

\begin{proposition}\label{Proposition: Reflected Bridge Formula}
For every $\be,T>0$ and $w\in\R$,
\begin{equation}\label{eq:exp_kernel}
\begin{split}
& \E\big[K^{\be;w}_T(0,0)\big] \\
& =\sqrt{\frac{2}{\pi T}}\,
\E\Bigg[\exp\Bigg(-\frac{T^{3/2}}2\bigg(\int_0^1r_t\d t-\int_0^\infty \frac{L^a_1(r)^2}\be\d a\bigg)-T^{1/2}w\frac{L^0_1(r)}2\Bigg)\Bigg],
\end{split}
\end{equation}
where $r_t$, $t\in[0,1]$ is a reflected Brownian bridge. 
\end{proposition}

Since the density of $L^0_1(r)$ is known (see e.g. \cite[equation (3)]{Pitman}), it suffices to find the conditional distribution of the functional 
\begin{equation}\label{eq:RBB_func}
\int_0^1 r_t\d t-\int_0^\infty \frac{L^a_1(r)^2}\be\d a
\end{equation}
given $L^0_1(r)$ to compute the right-hand side in \eqref{eq:exp_kernel}. For $\beta=2$, this leads to the following theorem of independent interest. 

\begin{theorem}\label{Theorem: Reflected Bridge-Local Time is Gaussian}
Let $r_t$, $t\in[0,1]$ be a reflected Brownian bridge. Then, for every $\al\ge 0$, the conditional distribution of the functional 
\begin{equation}\label{eq:RBB_func2}
\int_0^1 r_t\d t-\int_0^\infty \frac{L^a_1(r)^2}2\d a
\end{equation}
given $L^0_1(r)=\al$ is Gaussian with mean $-\al/4$ and variance $1/12$.
\end{theorem}

\begin{remark}
Conditional on $L^0_1(r)=0$, the process $r_t$, $t\in[0,1]$ is a standard Brownian excursion, so that Theorem \ref{Theorem: Reflected Bridge-Local Time is Gaussian} is a generalization of the distributional identity for the latter found in \cite[Corollary 2.15]{GS} (see also \cite[Theorem 1.1]{H}).  
\end{remark}

As a consequence of Theorem \ref{Theorem: Reflected Bridge-Local Time is Gaussian}, we obtain an explicit formula for $\E\big[K^{2;w}_T(0,0)\big]$. 

\begin{corollary}\label{Corollary: Expectation of Spiked Kernel at Zero}
For any $w\in\R$ and $T>0$, and with
\begin{equation}
C_{w;T}:=\frac{\sqrt{T}(T-4w)}{4\sqrt{2}},
\end{equation}
it holds
\begin{equation}\label{eq: kernel at 0 expl}
\E\big[K^{2;w}_T(0,0)\big]\\
=\sqrt{\frac{2}{\pi T}}\exp\left(\frac{T^3}{96}\right)\Big(1+\sqrt{\pi}C_{w;T}\exp\big(C_{w;T}^2\big)\big(\erf(C_{w;T})+1\big)\Big),
\end{equation}
where $\erf(z):=\frac{2}{\sqrt{\pi}}\int_0^z e^{-a^2}\d a$ denotes the error function.
\end{corollary}

For $\beta\neq 2$, we were not able to obtain an analogue of Theorem \ref{Theorem: Reflected Bridge-Local Time is Gaussian}. 

\begin{problem}\label{Problem: Joint Laplace}
Find the conditional distribution of the functional in \eqref{eq:RBB_func} given $L^0_1(r)$ for all $\beta>0$.
\end{problem}

\begin{remark}
A simple computation based on Theorem \ref{Theorem: Reflected Bridge-Local Time is Gaussian} shows that the \textit{unconditional} distribution of the random variable
\begin{equation}
A:=\sqrt{12}\,\left(\int_0^1r_{t}\d t-\int_0^\infty \frac{L^a_1(r)^2}2\d a\right)
\end{equation}
has a moment-generating function given by 
\begin{equation}\label{Equation: RBB Laplace}
\E\big[\exp(\kappa A)\big]=\exp\left(\frac{\kappa^2}{2}\right)-\sqrt{\frac{3}{2}}\,\kappa\,\exp\left(2 \kappa^2\right) \erfc\bigg(\sqrt{\frac{3}{2}}\,\kappa\bigg),
\end{equation}
where $\erfc(z):=\frac{2}{\sqrt{\pi}}\int_z^\infty e^{-a^2} \d a$ denotes the complementary error function. Therefore, it seems natural to view the density of $A$ as a sum of the standard Gaussian density and a function that integrates to $0$, which yields a corresponding decomposition of the moments of $A$ (see Table \ref{Table: Some Moments} for the first few moments). In particular, Table \ref{Table: Some Moments} suggests the following formula for the odd moments of $A$: 
\begin{equation}
\E\big[A^{2n-1}\big]=-\frac{2^n (2 n - 1)!}{4 (n - 1)!}\,\sqrt{6\pi},\quad n=1,\,2,\,\ldots,
\end{equation}
leading us to believe that $A$ admits an interesting combinatorial interpretation.

\begin{table}
\caption{The first few moments of $A$}
\begin{center}
\begin{tabular}{|c||c|}
\hline\vspace{-3.5mm}&\\
$\E[A]$&$-\sqrt{6 \pi }/2$\\
\vspace{-4mm}&\\\hline\vspace{-3.5mm}&\\
$\E[A^2]$&$1+6$\\
\vspace{-4mm}&\\\hline\vspace{-3.5mm}&\\
$\E[A^3]$&$-6 \sqrt{6 \pi }$\\
\vspace{-4mm}&\\\hline\vspace{-3.5mm}&\\
$\E[A^4]$&$3+108$\\
\vspace{-4mm}&\\\hline\vspace{-3.5mm}&\\
$\E[A^5]$&$-120 \sqrt{6 \pi }$\\
\vspace{-4mm}&\\\hline\vspace{-3.5mm}&\\
$\E[A^6]$&$15+2646$\\
\vspace{-4mm}&\\\hline\vspace{-3.5mm}&\\
$\E[A^7]$&$-3360  \sqrt{6 \pi }$\\
\hline
\end{tabular}
\begin{tabular}{|c||c|}
\hline\vspace{-3.5mm}&\\
$\E[A^8]$&$105+85032$\\
\vspace{-4mm}&\\\hline\vspace{-3.5mm}&\\
$\E[A^9]$&$-120960 \sqrt{6 \pi }$\\
\vspace{-4mm}&\\\hline\vspace{-3.5mm}&\\
$\E[A^{10}]$&$945+3404430$\\
\vspace{-4mm}&\\\hline\vspace{-3.5mm}&\\
$\E[A^{11}]$&$-5322240\sqrt{6 \pi }$\\
\vspace{-4mm}&\\\hline\vspace{-3.5mm}&\\
$\E[A^{12}]$&$10395+163446660$\\
\vspace{-4mm}&\\\hline\vspace{-3.5mm}&\\
$\E[A^{13}]$&$-276756480\sqrt{6 \pi }$\\
\vspace{-4mm}&\\\hline\vspace{-3.5mm}&\\
$\E[A^{14}]$&$135135+9153449550$\\
\hline
\end{tabular}
\end{center}
\label{Table: Some Moments}
\end{table}
\end{remark}

The remainder of the paper is structured as follows. In Section \ref{sec:strong_inv}, we prove a strong invariance principle for certain non-negative random walks and their families of occupation times, which plays a central role in the proof of Theorem \ref{Theorem: Main}. The proof of the invariance principle is based on the observation that the non-negative random walks in consideration can be viewed as images of simple random walks under the Skorokhod reflection map, and we expect the same idea to apply for a wide variety of constrained discrete processes. Section \ref{Section: Main 1} is devoted to the proof of Theorem \ref{Theorem: Main}. The proof has the following general structure: first, we write the entries of the matrix $\M^{\be;w}_{T;N}$ as expectations of suitable functionals of the non-negative random walks of Section \ref{sec:strong_inv} and the entries of the matrix $M_N^{\be;w}$; then, we derive the limiting behavior of suitably truncated versions of such expectations using the strong invariance principle of Section \ref{sec:strong_inv}; finally, we remove the truncation by obtaining appropriate uniform moment estimates on the functionals involved. In Section \ref{sec:properties}, we show the properties of the limiting operators $\U^{\be;w}_T$, $T\geq0$ listed in Proposition \ref{Proposition: Properties}. Lastly, in Section \ref{sec:Hariya} we establish Theorem \ref{Theorem: Reflected Bridge-Local Time is Gaussian}, as well as Proposition \ref{Proposition: Reflected Bridge Formula} and Corollary \ref{Corollary: Expectation of Spiked Kernel at Zero}. The proof of Theorem \ref{Theorem: Reflected Bridge-Local Time is Gaussian} combines the ideas of \cite{H} with the analogue of Jeulin's theorem for a reflected Brownian bridge conditioned on its local time at zero from \cite[Corollary 16(iii)]{Pitman2}.

\section{A strong invariance principle} \label{sec:strong_inv}

This section focuses on the strong invariance principle for certain non-negative random walks and their families of occupation times, which is at the heart of the proof of Theorem \ref{Theorem: Main}. For starters, we let $Y=(Y_0,\,Y_1,\,\ldots)$ be a random walk on the non-negative integers with transition probabilities 
\begin{equation}
\begin{split}
& \Pr[Y_{n+1}=z+1\,|\,Y_n=z]=\Pr[Y_{n+1}=z-1\,|\,Y_n=z]=\frac12,\quad z=1,\,2,\,\ldots, \\
&\Pr[Y_{n+1}=1\,|\,Y_n=0]=\Pr[Y_{n+1}=0\,|\,Y_n=0]=\frac12.
\end{split}
\end{equation}
In other words, when $Y$ is away from $0$, it behaves like a simple symmetric random walk (SSRW), and when $Y$ is at $0$, it stays at $0$ or moves to $1$ with equal probability.

\medskip

Given $T>0$ and $N\in\N$, we let $k=k(T,N):=\lfloor TN^{2/3} \rfloor$ and $T_k:=kN^{-2/3}$. Moreover, for each $x\ge0$, we define a process $X^{k;x}_t$, $t\in[0,T_k]$ satisfying
\begin{equation}\label{Equation: Lazy Walk Interpolation Points}
\big(X^{k;x}_{0},\,X^{k;x}_{N^{-2/3}},\,\ldots,\,X^{k;x}_{T_k}\big)\deq(Y_0,\,Y_1,\,\ldots,\,Y_k\,|\,Y_0=\lfloor xN^{1/3}\rfloor)
\end{equation}
and interpolating linearly between these time points (see Figure \ref{Figure: Lazy Path} for an illustration). We also introduce the normalized occupation times of $X^{k;x}$ for positive levels:  
\begin{equation}\label{Equation: occ_time>0}
L^a(X^{k;x}):=N^{-1/3}\big|\{t\in[0,T_k]:\,X^{k;x}=aN^{1/3}\}\big|,\quad a>0
\end{equation}
and use the convention
\begin{equation}\label{Equation: occ_time=0}
L^0(X^{k;x}):=\lim_{a\downarrow0}L^a(X^{k;x}).
\end{equation}
Finally, we let
\begin{equation}\label{Equation: horizonal steps}
H(X^{k;x}):=\big|\{t\in [0,T_k]\cap N^{-2/3}\N:\,X^{k;x}_{t-N^{-2/3}}=X^{k;x}_t=0\}\big|
\end{equation}
be the number of the horizontal steps at zero in $\big(X^{k;x}_{0},\,X^{k;x}_{N^{-2/3}},\,\ldots,\,X^{k;x}_{T_k}\big)$. Our strong invariance principle can now be stated as follows. 

\begin{figure}[htbp]
\begin{center}
%
%
\begin{tikzpicture}

\begin{axis}[%
xticklabels={,,},
width=4in,
height=1in,
at={(1.011in,0.642in)},
scale only axis,
xmin=0,
xmax=28,
ymin=-1,
ymax=5,
axis background/.style={fill=white}
]
\addplot [ forget plot, thick]
  table[row sep=crcr]{%
0	3\\
1	2\\
2	3\\
3	4\\
4	5\\
5	4\\
6	5\\
7	4\\
8	3\\
9	2\\
10	1\\
11	2\\
12	1\\
13	0\\
};
\addplot [ forget plot, thick]
  table[row sep=crcr]{%
13	0\\
14	0\\
};
\addplot [ forget plot, thick]
  table[row sep=crcr]{%
14	0\\
15	1\\
16	2\\
17	1\\
18	0\\
19	1\\
20	0\\
};
\addplot [ forget plot, thick]
  table[row sep=crcr]{%
20	0\\
21	0\\
};
\addplot [ forget plot, thick]
  table[row sep=crcr]{%
21	0\\
22	1\\
23	0\\
};
\addplot [ forget plot, thick]
  table[row sep=crcr]{%
23	0\\
24	0\\
};
\addplot [ forget plot, thick]
  table[row sep=crcr]{%
24	0\\
25	1\\
26	2\\
27	3\\
28	2\\
};
\end{axis}
\end{tikzpicture}%
\caption{A sample path of $X^{k;x}$ with $k=28$ and $\lfloor xN^{1/3}\rfloor=3$.}
\label{Figure: Lazy Path}
\end{center}
\end{figure}
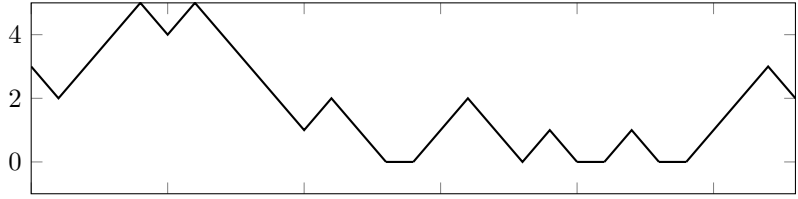

\begin{theorem}\label{Theorem: Coupling}
For every $T>0$ and $x\geq0$, there exists a coupling of the sequence of processes $X^{k;x}$, $N\in\N$ and a reflected Brownian motion $R^x$ such that 
\begin{eqnarray}
&& \sup_{t\in[0,T_k]} \big|N^{-1/3}X^{k;x}_t-R^x_t\big| \leq \C N^{-1/3}\log N,\quad N\in\N, \label{Equation: Coupling of Lazy Walks 1} \\
&& \sup_{a>0}\,\big|L^a(X^{k;x})-L^a_T(R^x)\big| \leq \C N^{-1/16},\quad N\in\N, \label{Equation: Coupling of Lazy Walks 2} \\
&& \big|N^{-1/3}H(X^{k;x})-L^0_T(R^x)/2\big|\leq \C N^{-1/3}\log N,\quad N\in\N, \label{Equation: Coupling of Lazy Walks 3}
\end{eqnarray}
where $\C$ is a suitable finite random variable. 
\end{theorem}

A direct construction of the coupling in Theorem \ref{Theorem: Coupling} appears to be difficult. Instead, our proof of Theorem \ref{Theorem: Coupling} relies on the Koml\'os-Major-Tusn\'ady coupling of a SSRW with a standard Brownian motion and an application of the Skorokhod reflection map. We briefly recall the definition and some properties of the latter. 

\begin{definition}
Given a $T>0$ and a continuous process $Z_t$, $t\in[0,T]$, we define the Skorokhod map evaluated at $Z$ as the continuous process 
\begin{equation}
\Gamma(Z)_t=Z_t+\sup_{s\in[0,t]}\,(-Z_s)_+,\quad t\in[0,T],
\end{equation}
where $(\cdot)_+:=\max(0,\cdot)$ denotes the positive part of a real number.
\end{definition}

A reflected Brownian motion $R^x$ can be defined by $R^x=|x+\widetilde W|$, where $\widetilde W$ is a standard Brownian motion. According to Tanaka's formula (see e.g. \cite[Chapter VI, Theorem 1.2]{RY}) one has
\begin{equation}
R^x_t=x+\int_0^t\sgn(\widetilde W_s)\d \widetilde W_s+L_t^{-x}(\widetilde W),\quad t\in[0,T].
\end{equation}
Furthermore, if we let
\begin{equation}
B^x_t:=x+\int_0^t\sgn(\widetilde W_s)\d \widetilde W_s,\quad t\in[0,T],
\end{equation}
which is a Brownian motion started at $x$, then it follows from a classical result of Skorokhod \cite[Chapter VI, Lemma 2.1 and Corollary 2.2]{RY} that
\begin{eqnarray}
&& R^x_t=B^x_t+\sup_{s\in[0,t]}\,(-B^x_s)_+=\Gamma(B^x)_t,\quad t\in[0,T], \label{eq:refl_map} \\
&& \sup_{s\in[0,t]}\,(-B^x_s)_+=L_t^{-x}(\widetilde W)=L_t^{0}(x+\widetilde W)=\frac{L_t^{0}(R^x)}2,\quad t\in[0,T]. \label{eq:loc_time}
\end{eqnarray}
We give in the next proposition a discrete analogue of these results.

\begin{proposition}\label{Proposition: Discrete Skorokhod}
Let $\widetilde Y=(\widetilde Y_0,\,\widetilde Y_1,\,\ldots)$ be a SSRW
and define the process $\widetilde X^{k;x}$, $t\in[0,T_k]$ by using the same procedure as for $X^{k;x}$ (that is, equation \eqref{Equation: Lazy Walk Interpolation Points} followed by a linear interpolation),
but with with the SSRW $\widetilde Y$ instead of the random walk $Y$.
Then, $Y$ and $\widetilde Y$ can be coupled in such a way that
$X^{k;x}_t=\Gamma(\widetilde X^{k;x})_t$, $t\in[0,T_k]$ and $H(X^{k;x})=\sup_{t\in[0,T_k]}\,(-\widetilde X^{k;x}_t)_+$.
\end{proposition}

\begin{proof}
Both $X^{k;x}$ and $\widetilde X^{k;x}$ can take a total of $2^k$ possible sample paths, and the measures that $Y$ and $\widetilde Y$ induce on these paths are uniform in both cases. Therefore, we need to show that $\Gamma$ is a bijection taking paths of $\widetilde X^{k;x}$ to paths of $X^{k;x}$ and that $H(\Gamma(\widetilde X^{k;x}))=\sup_{t\in[0,T_k]}\,(-\widetilde X^{k;x}_t)_+$. 

\medskip

Whenever $\min\widetilde X^{k;x}\ge0$, the Skorokhod map $\Gamma$ leaves $\widetilde X^{k;x}$ unchanged, and trivially $H(\Gamma(\widetilde X^{k;x}))=0=\sup_{t\in[0,T_k]}\,(-\widetilde X^{k;x}_t)_+$. On the other hand, whenever $\min\widetilde X^{k;x}<0$, the application of $\Gamma$ can be described as follows (see Figure \ref{Figure: Skorokhod Bijection} below): 
\begin{enumerate}[(a)]
\item one determines the first hitting times $\tau_{-1}<\tau_{-2}<\cdots$ of the negative integer levels by $\widetilde X^{k;x}$;
\item one sets $\Gamma(\widetilde X^{k;x})_t:=\widetilde X^{k;x}_t$ for $t\in\big[0,\,\tau_{-1}-N^{-2/3}\big]$;
\item for $j=1,\,2,\,\ldots,\,-\min \widetilde X^{k;x}$, one lets $\Gamma(\widetilde X^{k;x})_t:=0$ for $t\in\big(\tau_{-j}-N^{-2/3},\,\tau_{-j}\big]$; 
\item for $j=1,\,2,\,\ldots,\,-\min \widetilde X^{k;x}-1$, one defines $\Gamma(\widetilde X^{k;x})_t:=\widetilde X^{k;x}_t+j$ for $t\in\big(\tau_{-j},\,\tau_{-j-1}-N^{-2/3}\big]$;
\item one puts $\Gamma(\widetilde X^{k;x})_t:=\widetilde X^{k;x}_t-\min \widetilde X^{k;x}$ for $t\in\big(\tau_{\min\widetilde X^{k;x}},\,T_k\big]$.
\end{enumerate}

\medskip

It follows immediately from this description that the horizontal steps at zero in $\big(\Gamma(\widetilde X^{k;x})_0,\,\Gamma(\widetilde X^{k;x})_{N^{-2/3}},\,\ldots,\,\Gamma(\widetilde X^{k;x})_{T_k}\big)$ occur at $\tau_{-1}-N^{-2/3}$, $\tau_{-2}-N^{-2/3}$, $\ldots$, $\tau_{\min\widetilde X^{k;x}}-N^{-2/3}$, so that $H(\Gamma(\widetilde X^{k;x}))=-\min\widetilde X^{k;x}=\sup_{t\in[0,T_k]}\,(-\widetilde X^{k;x}_t)_+$. Moreover, for every path of $X^{k;x}$, one can uniquely reconstruct the corresponding path $\Gamma^{-1}(X^{k;x})=\widetilde X^{k;x}$ by inferring the sequence $\tau_{-1}<\tau_{-2}<\cdots$ from the horizontal segments in the path of $X^{k;x}$, solving the equations in (b), (d), (e) above, and inserting the remaining $H(X^{k;x})$ downward sloping segments.  
\end{proof}

\begin{figure}[htbp]
\begin{center}
%
%
\begin{tikzpicture}

\begin{axis}[%
xticklabels={,,},
width=4in,
height=1in,
at={(1.011in,0.642in)},
scale only axis,
xmin=0,
xmax=28,
ymin=-3,
ymax=5,
axis background/.style={fill=white}
]
\addplot [ forget plot, thick]
  table[row sep=crcr]{%
0	3\\
1	2\\
2	3\\
3	4\\
4	5\\
5	4\\
6	5\\
7	4\\
8	3\\
9	2\\
10	1\\
11	2\\
12	1\\
13	0\\
};
\addplot [dotted, forget plot, thick]
  table[row sep=crcr]{%
13	0\\
14	-1\\
};
\addplot [ forget plot, thick]
  table[row sep=crcr]{%
14	-1\\
15	0\\
16	1\\
17	0\\
18	-1\\
19	0\\
20	-1\\
};
\addplot [dotted, forget plot, thick]
  table[row sep=crcr]{%
20	-1\\
21	-2\\
};
\addplot [ forget plot, thick]
  table[row sep=crcr]{%
21	-2\\
22	-1\\
23	-2\\
};
\addplot [dotted, forget plot, thick]
  table[row sep=crcr]{%
23	-2\\
24	-3\\
};
\addplot [ forget plot, thick]
  table[row sep=crcr]{%
24	-3\\
25	-2\\
26	-1\\
27	0\\
28	-1\\
};
\end{axis}
\draw (4.4,2) node{$\Gamma^{-1}(X^{k;x})$ or $\widetilde X^{k;x}$};
\end{tikzpicture}%
%
%
\begin{tikzpicture}

\begin{axis}[%
xticklabels={,,},
width=4in,
height=1in,
at={(1.011in,0.642in)},
scale only axis,
xmin=0,
xmax=28,
ymin=-3,
ymax=5,
axis background/.style={fill=white}
]
\addplot [ forget plot, thick]
  table[row sep=crcr]{%
0	3\\
1	2\\
2	3\\
3	4\\
4	5\\
5	4\\
6	5\\
7	4\\
8	3\\
9	2\\
10	1\\
11	2\\
12	1\\
13	0\\
};
\addplot [dotted, forget plot, thick]
  table[row sep=crcr]{%
13	0\\
14	0\\
};
\addplot [ forget plot, thick]
  table[row sep=crcr]{%
14	0\\
15	1\\
16	2\\
17	1\\
18	0\\
19	1\\
20	0\\
};
\addplot [dotted, forget plot, thick]
  table[row sep=crcr]{%
20	0\\
21	0\\
};
\addplot [ forget plot, thick]
  table[row sep=crcr]{%
21	0\\
22	1\\
23	0\\
};
\addplot [dotted, forget plot, thick]
  table[row sep=crcr]{%
23	0\\
24	0\\
};
\addplot [ forget plot, thick]
  table[row sep=crcr]{%
24	0\\
25	1\\
26	2\\
27	3\\
28	2\\
};
\end{axis}
\draw (3.3,2) node{$X^{k;x}$};
\end{tikzpicture}%
\caption{Illustration of the discrete Skorokhod map.}
\label{Figure: Skorokhod Bijection}
\end{center}
\end{figure}
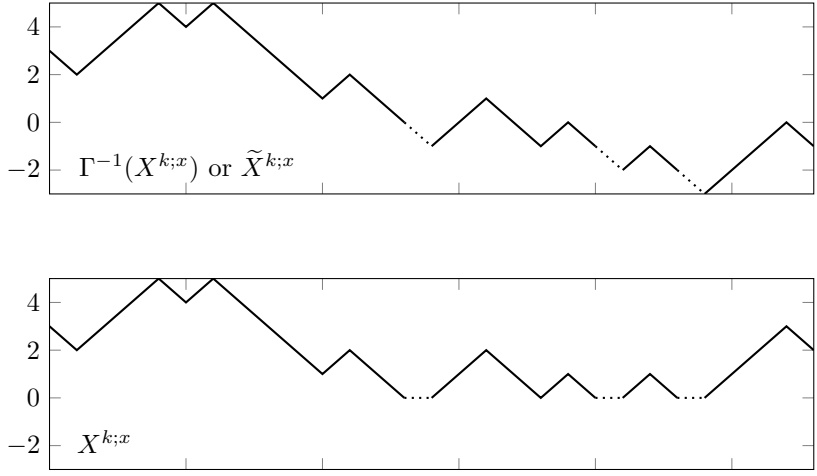

Next, we prepare another coupling needed for the proof of Theorem \ref{Theorem: Coupling}.

\begin{lemma}\label{Lemma: Coupling of Simple Walks}
For every $T>0$ and $x\geq0$, there exists a coupling of the sequence of processes $\widetilde X^{k;x}$, $N\in\N$ defined in Proposition \ref{Proposition: Discrete Skorokhod} and a standard Brownian motion $B$ such that
\begin{equation}\label{Equation: Coupling of Simple Walks 1}
\sup_{t\in[0,T_k]} \big|N^{-1/3}\widetilde X^{k;x}_t-(x+B_t)\big| \leq \C N^{-1/3}\log N,\quad N\in\N,
\end{equation}
where $\C$ is a suitable finite random variable. 
\end{lemma}

\begin{proof}
Consider a probability space which supports a standard Brownian motion $B$ and define the standard Brownian motions  
\begin{equation}
B^{(N)}_t:=N^{1/3}B_{tN^{-2/3}},\quad t\geq0
\end{equation}
for all $N\in\N$. According to a well-known procedure of Koml\'os, Major and Tusn\'ady (see e.g. \cite[Section 7]{LL}), one can construct random walks $(\widetilde Y^{(N)}_0,\,\widetilde Y^{(N)}_1,\,\ldots)$ , $N\in\N$ as deterministic functions of the Brownian motions $B^{(N)}$, $N\in\N$, respectively, such that for every $\al>0$, there exists a $C<\infty$ so that
\begin{equation}
\Pr\Big[\max_{0\leq n\leq k} \big|\widetilde Y^{(N)}_n-B^{(N)}_n\big|\geq C\log N\Big]\leq CN^{-\al},\quad N\in\N. 
\end{equation}
As a result, we can couple the sequence $\widetilde X^{k;x}$, $N\in\N$ with $B$ ensuring
\begin{equation}
\Pr\Big[\max_{0\leq i\leq k} \big|N^{-1/3}X^{k;x}_{N^{-2/3}i}-B_{N^{-2/3}i}\big|\geq C N^{-1/3}\log N\Big]\leq CN^{-\al},\quad N\in\N.
\end{equation}
Lastly, we let $\alpha>1$ and conclude by applying the Borel-Cantelli lemma and the L\'evy modulus of continuity theorem.
\end{proof}

Finally, we define the random walk $\bar Y=(\bar Y_0,\,\bar Y_1,\,\ldots)$ on the integers, to be used in the proof of Theorem \ref{Theorem: Coupling}, by 
\begin{equation}\label{eq:whatisYbar}
\begin{split}
& \Pr[\bar Y_{n+1}=z+1\,|\,\bar Y_n=z]=\Pr[\bar Y_{n+1}=z-1\,|\,\bar Y_n=z]=\frac12,\quad z\in\Z\backslash\{0\}, \\
&\Pr[\bar Y_{n+1}=1\,|\,\bar Y_n=0]=\Pr[\bar Y_{n+1}=-1\,|\,\bar Y_n=0]=\frac14,\quad \Pr[\bar Y_{n+1}=0\,|\,\bar Y_n=0]=\frac12.
\end{split}
\end{equation}
For every $n\in\N$, we let
\begin{equation}
H_n(\bar Y)=\big|\{1\le i\leq n:\,\bar Y_{i-1}=\bar Y_i=0\}\big|
\end{equation}
be the number of the horizontal steps at zero among the first $n$ steps of $\bar Y$ and define $v(\bar Y)_n$, $n=0,\,1,\,\ldots$ as the process obtained from $\bar Y$ by removing all the horizontal steps at zero, so that
\begin{equation}\label{Equation: Lazy Path Random Shift}
\bar Y_n=v(\bar Y)_{n-H_n(\bar Y)},\quad n=0,\,1,\,\ldots.
\end{equation}
We also introduce, for $m\in\N$, $T>0$ and $a\neq0$, the normalized occupation times 
\begin{equation}\label{eq:loc_time_norm}
L_{m;T}^a(\bar Y):=m^{-1/2}\big|\{0\le t\le \lfloor mT\rfloor:\,\bar Y_t=\sqrt m a\}\big|,
\end{equation}
where we define $\bar Y$ for non-integer times by linear interpolation. Lastly, we let $L_{m;T}^a(v(\bar Y))$ be given by \eqref{eq:loc_time_norm} with $v(\bar Y)$ in place of $\bar Y$. 

\begin{remark}\label{Remark: Lazy Path 2}
By examining the transition probabilities of $\bar Y$ it becomes clear that $v(\bar Y)_n$, $n=0,\,1,\,\ldots$ is a SSRW, and that $Y$ and $\bar Y$ can be coupled to obey
\begin{equation}
|\bar Y_n|=Y_n,\quad n=0,\,1,\,\ldots,
\end{equation}
provided the two processes have the same starting point.
If we condition on the starting point $Y_0=\bar Y_0=\lfloor N^{1/3}x\rfloor$,
then it holds under this coupling that
\begin{equation}\label{Equation: Lazy Local Time}
L^a(X^{k;x})=L^a_{N^{2/3};T}(\bar Y)+L^{-a}_{N^{2/3};T}(\bar Y),\quad a>0.
\end{equation}
\end{remark}

We conclude the section with the proof of Theorem \ref{Theorem: Coupling}.

\begin{proof}[Proof of Theorem \ref{Theorem: Coupling}]
The maps $\Gamma$ and $f\mapsto \sup_{t\in[0,T]} (-f(t))_+$ are $2$-Lipschitz and $1$-Lipschitz with respect to the supremum norm, respectively, so that \eqref{Equation: Coupling of Lazy Walks 1} and \eqref{Equation: Coupling of Lazy Walks 3} follow from \eqref{Equation: Coupling of Simple Walks 1} by combining Proposition \ref{Proposition: Discrete Skorokhod} with \eqref{eq:refl_map} and \eqref{eq:loc_time}, respectively. 

\medskip

The remaining estimate \eqref{Equation: Coupling of Lazy Walks 2} is a consequence of the estimate \eqref{Equation: Coupling of Lazy Walks 1} and the regularity of the local time processes involved. More specifically, by using the same arguments as in \cite[Appendix B]{GS} 
(note that $a\mapsto L^a_T(R^x)$ inherits the regularity properties of $a\mapsto L^a_T(\widetilde{W})$ due to $R^x=|x+\widetilde W|$), we can reduce \eqref{Equation: Coupling of Lazy Walks 2} to the following statement: for every $\eps>0$, there exists a finite random variable $\C_\eps$ such that
\begin{equation}\label{claim:loc_time_reg}
\sup_{\substack{a_1,a_2>0\\|a_1-a_2|\leq N^{-2/15}}}\big|L^{a_1}(X^{k;x})-L^{a_2}(X^{k;x})\big|\leq\C_\eps N^{-1/15+\eps},\quad N\in\N
\end{equation}
(cf. \cite[Lemma B.3]{GS}). To this end and in view of \eqref{Equation: Lazy Local Time}, it suffices to prove that
\begin{equation}\label{claim:loc_time_reg2}
\sup_{\substack{a_1,a_2\in\R\setminus\{0\}\\|a_1-a_2|\leq N^{-2/15}}} \big|L_{N^{2/3};T}^{a_1}(\bar Y)-L_{N^{2/3};T}^{a_2}(\bar Y)\big|\leq\C_\eps N^{-1/15+\eps},\quad N\in\N.
\end{equation}
Applying \cite[Proposition 3.1]{BK} to the SSRW $v(\bar Y)$ we get
\begin{equation}
\sup_{t\in[0,T]} \sup_{\substack{a_1,a_2\in\R\setminus\{0\}\\|a_1-a_2|\leq N^{-2/15}}}\big|L_{N^{2/3};t}^{a_1}\big(v(\bar Y)\big)-L_{N^{2/3};t}^{a_2}\big(v(\bar Y)\big)\big|\leq\C_\eps N^{-1/15+\eps},\quad N\in\N,
\end{equation}
at which point the desired estimate \eqref{claim:loc_time_reg2} follows from \eqref{Equation: Lazy Path Random Shift}.
\end{proof}

\section{Proof of Theorem \ref{Theorem: Main}}\label{Section: Main 1}

This section is devoted to the proof of Theorem \ref{Theorem: Main}. Since the proof is rather long, we first present an informal overview of the arguments in Subsection \ref{Section: Strategy Main 1}, before rigorously carrying out the proof in Subsections \ref{Section:truncated}--\ref{Section: Final Proof Synthesis}.

\subsection{Informal overview of the proof} \label{Section: Strategy Main 1}

Let $\beta>0$, $w\in\R$, $T>0$ and $f,g\in{\mathcal D}$ be fixed. Our object of study is the scalar product
\begin{equation}
\big(\pi_Nf\big)^\top\M^{\be;w}_{T;N}\big(\pi_Ng\big)
=\sum_{l,l'=0}^N \big(\pi_Nf\big)[l]\cdot\M^{\be;w}_{T;N}[l,l']\cdot\big(\pi_Ng\big)[l'].
\end{equation}
Here and throughout the paper, we index all $(N+1)\times(N+1)$ matrices $A$ by $l,l'\in\{0,\,1,\,\ldots,\,N\}$ and write $A[l,l']$ for the $(l,l')$-entry of $A$. Similarly, the entries of all $(N+1)$-dimensional vectors $v$ are denoted by $v[l]$ for $l\in\{0,\,1,\,\ldots,\,N\}$.
By the definition of $\pi_N$ in \eqref{Equation: Projection}, we then see that
\begin{eqnarray*}
&& \big(\pi_Nf\big)^\top\M^{\be;w}_{T;N}\big(\pi_Ng\big)=\int_0^{N^{2/3}}\int_0^{N^{2/3}} f(x)\,K^{\be;w}_{T;N}(x,y)\,g(y)\,\d x \,\dd y,\quad\text{where} \\
&& K^{\be;w}_{T;N}:=\sum_{l,l'=0}^N N^{1/3}\M^{\be;w}_{T;N}[l,l']
\,\1_{[N^{-1/3}l,N^{-1/3}(l+1))\times[N^{-1/3}l',N^{-1/3}(l'+1))}.
\end{eqnarray*}

Recalling $k=k(T,N)=\lfloor TN^{2/3}\rfloor$ and the definition of $\M^{\be;w}_{T;N}$ in \eqref{eq:script_M}, we find for all $l,l'\in\{0,\,1,\,\ldots,\,N\}$:
\begin{equation}\label{eq:matrix_entry}
\M^{\be;w}_{T;N}[l,l']=\frac1{(2\sqrt{N})^k}\sum_{0\leq l_1,\ldots,l_{k-1}\leq N} M^{\be;w}_N[l,l_1]\,M^{\be;w}_N[l_1,l_2]\,\cdots \,M^{\be;w}_N[l_{k-1},l'].
\end{equation}
Since $M^{\be;w}_N$ is tridiagonal, only $(k+1)$-tuples $(l_0,\,l_1,\,\ldots,\,l_k)$ that satisfy $l_0=l$, $l_k=l'$ and $|l_{s-1}-l_s|\in\{0,1\}$ for all $s$ contribute to $\M^{\be;w}_{T;N}[l,l']$. Any such $(k+1)$-tuple can be thought of as a path from $l_0$ to $l_k$ that takes steps of size $+1$ or $-1$ (when $|l_{s-1}-l_s|=1$), and horizontal steps (when $l_{s-1}=l_s$). In the following, we rely on this observation to write the sum on the right-hand side of \eqref{eq:matrix_entry} in terms of expectations with respect to the random walks of Section \ref{sec:strong_inv}.  

\medskip

For $j=0,\,1,\,\ldots$ and $x\ge0$, we define the random walk $X^{k-j;x}$, its normalized occupation times and its number of horizontal steps at zero by \eqref{Equation: Lazy Walk Interpolation Points}-\eqref{Equation: horizonal steps}, with $(k-j)$ in place of $k$. We also let $\hat X^{k-j;x}_t$, $t\in[0,T_{k-j}-N^{-2/3}H(X^{k-j;x})]$ be the path obtained from $X^{k-j;x}_t$, $t\in[0,T_{k-j}]$ by removing all horizontal segments at zero (see Figure \ref{Figure: Lazy Path Hat}). Finally, we introduce the functional 
\begin{equation}\label{Equation: Genuine Functional}
\begin{split}
& F_j(X^{k-j;x},\a,\xi):=\prod_{i=1}^{k-j-H(X^{k-j;x})}\frac{\sqrt{N-\hat X^{k-j;x}_{N^{-2/3}(i-1)}\land\hat X^{k-j;x}_{N^{-2/3}i}}}{\sqrt{N}} \\
&\qquad\qquad\qquad\quad\;\;\; \cdot\prod_{i=1}^{k-j-H(X^{k-j;x})} \Bigg(1+\frac{\xi_{\hat X^{k-j;x}_{N^{-2/3}(i-1)}\land\hat X^{k-j;x}_{N^{-2/3}i}}}{\sqrt{N-\hat X^{k-j;x}_{N^{-2/3}(i-1)}\land\hat X^{k-j;x}_{N^{-2/3}i}}}\Bigg)
\ell_N^{H(X^{k-j;x})}\\
&\qquad\qquad\qquad\quad\;\;\; \cdot\Bigg(\frac{1}{(2\sqrt{N})^j}\sum_{0\leq i_1\leq\cdots\leq i_j\leq k}\prod_{j'=1}^j\a_{\hat X^{k-j;x}_{N^{-2/3}i_{j'}}}\Bigg),
\end{split}
\end{equation}
where the random walk $X^{k-j;x}$ is independent of all $\a_m$'s and $\xi_m$'s.

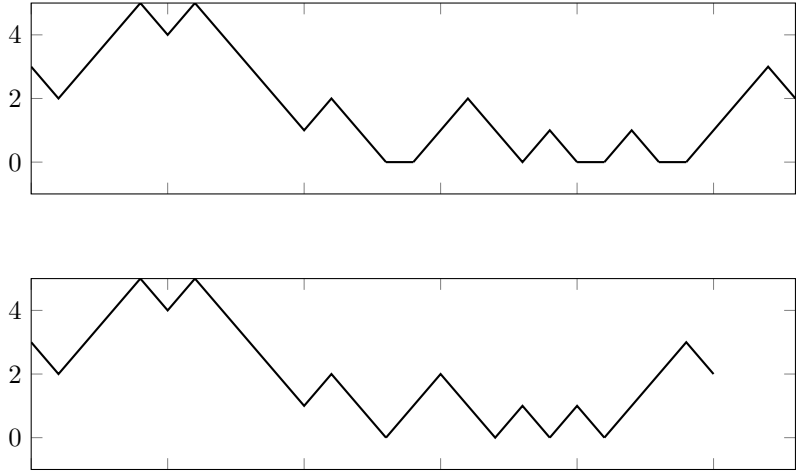
\begin{figure}[htbp]
\begin{center}
%
%
\begin{tikzpicture}

\begin{axis}[%
xticklabels={,,},
width=4in,
height=1in,
at={(1.011in,0.642in)},
scale only axis,
xmin=0,
xmax=28,
ymin=-1,
ymax=5,
axis background/.style={fill=white}
]
\addplot [ forget plot, thick]
  table[row sep=crcr]{%
0	3\\
1	2\\
2	3\\
3	4\\
4	5\\
5	4\\
6	5\\
7	4\\
8	3\\
9	2\\
10	1\\
11	2\\
12	1\\
13	0\\
};
\addplot [ forget plot, thick]
  table[row sep=crcr]{%
13	0\\
14	0\\
};
\addplot [ forget plot, thick]
  table[row sep=crcr]{%
14	0\\
15	1\\
16	2\\
17	1\\
18	0\\
19	1\\
20	0\\
};
\addplot [ forget plot, thick]
  table[row sep=crcr]{%
20	0\\
21	0\\
};
\addplot [ forget plot, thick]
  table[row sep=crcr]{%
21	0\\
22	1\\
23	0\\
};
\addplot [ forget plot, thick]
  table[row sep=crcr]{%
23	0\\
24	0\\
};
\addplot [ forget plot, thick]
  table[row sep=crcr]{%
24	0\\
25	1\\
26	2\\
27	3\\
28	2\\
};
\end{axis}
\end{tikzpicture}%
%
%
\begin{tikzpicture}

\begin{axis}[%
xticklabels={,,},
width=4in,
height=1in,
at={(1.011in,0.642in)},
scale only axis,
xmin=0,
xmax=28,
ymin=-1,
ymax=5,
axis background/.style={fill=white}
]
\addplot [ forget plot, thick]
  table[row sep=crcr]{%
0	3\\
1	2\\
2	3\\
3	4\\
4	5\\
5	4\\
6	5\\
7	4\\
8	3\\
9	2\\
10	1\\
11	2\\
12	1\\
13	0\\
};
\addplot [ forget plot, thick]
  table[row sep=crcr]{%
13	0\\
14	1\\
15	2\\
16	1\\
17	0\\
18	1\\
19	0\\
};
\addplot [ forget plot, thick]
  table[row sep=crcr]{%
19	0\\
20	1\\
21	0\\
};
\addplot [ forget plot, thick]
  table[row sep=crcr]{%
21	0\\
22	1\\
23	2\\
24	3\\
25	2\\
};
\end{axis}
\end{tikzpicture}%
\caption{Realizations of $X^{k-j;x}$ (top) and $\hat X^{k-j;x}$ (bottom).}
\label{Figure: Lazy Path Hat}
\end{center}
\end{figure}

\medskip

If $x\in[N^{-1/3}l,N^{-1/3}(l+1))$ and $y\in[N^{-1/3}l',N^{-1/3}(l'+1))$, then by definition of $M^{\be;w}_N$, one has
\begin{equation}
\M^{\be;w}_{T;N}[l,l']=\sum_{j=0}^k\frac{Q^{x,y}_{k-j}}{2^{k-j}}\,\E_{X^{k-j;x}}\big[F_{j}(X^{k-j;x},\a,\xi)\,\big|\,X^{k-j;x}_{T_{k-j}}=\lfloor N^{1/3}y\rfloor\big],
\end{equation}
with $Q_{k-j}^{x,y}$ being the number of paths $X^{k-j;x}$ can take such that $X^{k-j;x}_{T_{k-j}}=\lfloor N^{1/3}y\rfloor$,
or equivalently,
\begin{equation}
Q_{k-j}^{x,y}:=2^{k-j}\,\Pr\big[X^{k-j;x}_{T_{k-j}}=\lfloor N^{1/3}y\rfloor\big].
\end{equation}
In the above, the parameter $j$ represents the number of times it holds $l_{s-1}=l_s\neq0$ within a $(k+1)$-tuple $(l_0,\,l_1,\,\ldots,\,l_k)$. Removing the corresponding horizontal steps from the associated path leaves us with a path of $X^{k-j;x}$. At the same time, the term in the third line of \eqref{Equation: Genuine Functional} keeps track of all the possible ways $j$ horizontal steps away from zero can be inserted into a given realization of $X^{k-j;x}$. Finally, the term $Q_{k-j}^{x,y}$ arises from the normalization inherent to the conditional expectation $\E_{X^{k-j;x}}[\,\cdot\,|\,X^{k-j;x}_{T_{k-j}}=\lfloor N^{1/3}y\rfloor]$.

\medskip

Next, we let
\begin{equation}
\begin{split}
& \Sc^j_k(f,g):=
\int_0^{N^{2/3}}\int_0^{N^{2/3}}N^{1/3}f(x) \\
&\qquad\qquad\qquad\quad\;\;\cdot\frac{Q_{k-j}^{x,y}}{2^{k-j}}\,
\E_{X^{k-j;x}}\big[F_{j}(X^{k-j;x},\a,\xi)\,\big|\,X^{k-j;x}_{T_{k-j}}=\lfloor N^{1/3}y\rfloor\big]\,g(y)\d x\d y,
\end{split}
\end{equation}
$j=0,\,1,\,\ldots$ and observe
\begin{equation}\label{Equation: Weak Convergence Prelimit}
\big(\pi_Nf\big)^\top\M^{\be;w}_{T;N}\big(\pi_Ng\big)=\sum_{j=0}^k\Sc^j_k(f,g).
\end{equation}
We also note that, by the total probability rule,
\begin{equation}
\Sc^j_k(f,g)=
\int_0^{N^{2/3}} \!\!\!\! f(x)\,\E_{X^{k-j;x}}\bigg[F_j(X^{k-j;x},\a,\xi)
\cdot N^{1/3}\!\int_{N^{-1/3}X^{k-j;x}_{T_{k-j}}}^{N^{-1/3}(X^{k-j;x}_{T_{k-j}}+1)} \!\! g(y)\!\!\d y\bigg]\!\!\!\d x,
\end{equation}
$j=0,\,1,\,\ldots$. The proof of Theorem \ref{Theorem: Main} now hinges on justifying the following heuristic computation.

\begin{heuristic}\label{Heuristic: Weak Convergence}
We recall the strong invariance principle of Theorem \ref{Theorem: Coupling}. Since $\log(1+z)=z+O(z^2)$ when $z\approx0$, we have for $j=0,\,1,\,\ldots$ that
\begin{equation}\label{Equation: H Area Approximation}
\begin{split}
&\prod_{i=1}^{k-j-H(X^{k-j;x})}
\frac{\sqrt{N-\hat X^{k-j;x}_{N^{-2/3}(i-1)}\land\hat X^{k-j;x}_{N^{-2/3}i}}}{\sqrt{N}} \\
& =\exp\Bigg(\frac12 \sum_{i=1}^{k-j-H(X^{k-j;x})}
\log\Bigg(1-\frac{\hat X^{k-j;x}_{N^{-2/3}(i-1)}\land\hat X^{k-j;x}_{N^{-2/3}i}}{N}\Bigg)\Bigg) \\
&\approx\exp\Bigg(-\frac{1}{2N^{2/3}}\sum_{i=1}^{k-j-H(X^{k-j;x})} \Bigg(\frac{\hat X^{k-j;x}_{N^{-2/3}(i-1)}\land\hat X^{k-j;x}_{N^{-2/3}i}}{N^{1/3}}\Bigg)\Bigg)\\
& \to\exp\bigg(-\int_0^T\frac{R^{x}_t}2\d t\Bigg),\quad N\to\infty.
\end{split}
\end{equation}
At the same time, $(1-z)^{-1/2}=1+O(z)$ when $z\approx0$ suggests for $j=0,\,1,\,\ldots$ that
\begin{equation}\label{Equation: H xi Approximation 0}
\begin{split}
&\prod_{i=1}^{k-j-H(X^{k-j;x})}\Bigg(1+\frac{\xi_{\hat X^{k-j;x}_{N^{-2/3}(i-1)}\land\hat X^{k-j;x}_{N^{-2/3}i}}}{\sqrt{N-\hat X^{k-j;x}_{N^{-2/3}(i-1)}\land\hat X^{k-j;x}_{N^{-2/3}i}}}\Bigg) \\
&=\prod_{i=1}^{k-j-H(X^{k-j;x})}\!\!\Bigg(\!1\!+\!\frac{\xi_{\hat X^{k-j;x}_{N^{-2/3}(i-1)}\land\hat X^{k-j;x}_{N^{-2/3}i}}}{\sqrt{N}}\Bigg(\!1\!-\!\frac{\hat X^{k-j;x}_{N^{-2/3}(i-1)}\!\land\!\hat X^{k-j;x}_{N^{-2/3}i}}{N}\Bigg)^{-1/2}\Bigg) \\
&\approx\exp\Bigg(\sum_{i=1}^{k-j-H(X^{k-j;x})} \log\Bigg(1+\frac{\xi_{\hat X^{k-j;x}_{N^{-2/3}(i-1)}\land\hat X^{k-j;x}_{N^{-2/3}i}}}{\sqrt{N}}\Bigg)\Bigg) \\
&\approx\exp\Bigg(\sum_{i=1}^{k-j-H(X^{k-j;x})}\;\frac{\xi_{\hat X^{k-j;x}_{N^{-2/3}(i-1)}\land\hat X^{k-j;x}_{N^{-2/3}i}}}{\sqrt{N}}\Bigg) \\
&=\exp\bigg(\sum_{a\in N^{-1/3}(\N-1/2)} L^a(X^{k-j;x})\,\frac{\xi_{\lfloor N^{1/3}a\rfloor}}{N^{1/6}}\bigg) 
\to\exp\bigg(s_\xi\int_0^\infty L^a_T(R^{x})\d W^\xi_a\bigg),
\end{split}
\end{equation}
as $N\to\infty$, where $W^\xi$ is the Brownian motion arising from a Donsker type invariance principle for the sequence $\xi_0,\,\xi_1,\,\ldots$ (see \eqref{Equation: B.M.s Joint with Truncated} below). Moreover, 
\begin{equation}
\begin{split}
\ell_N^{H(X^{k-j;x})}&=\left(1-\frac{N^{1/3}(1-\ell_N)}{N^{1/3}}\right)^{N^{1/3}\cdot(H(X^{k-j;x})/N^{1/3})}\\
& \to\exp\left(-w\frac{L^0_T(R^{x})}2\right),\quad N\to\infty,
\end{split}
\end{equation}
for $j=0,\,1,\,\ldots$.

\medskip

Next, we consider $j=2$ and make the simple observation
\begin{equation}\label{Equation: Newton trick}
\frac1N\sum_{0\le i_1\leq i_2\leq k}\a_{\hat X^{k-2;x}_{N^{-2/3}i_{1}}}\a_{\hat X^{k-2;x}_{N^{-2/3}i_{2}}} 
=\frac1{2N}\bigg(\sum_{i=0}^k \a_{\hat X^{k-2;x}_{N^{-2/3}i}}\bigg)^2+\frac1{2N}\sum_{i=0}^k \a_{\hat X^{k-2;x}_{N^{-2/3}i}}^2.
\end{equation}
In addition, from $k=O(N^{2/3})$ and Assumption \ref{Assumption: Matrix Assumption} we infer that the second summand on the right-hand side of \eqref{Equation: Newton trick} is negligible in the limit $N\to\infty$. Similar reasoning for $j=3,\,4,\,\ldots$ reveals that, for all $j=1,\,2,\,\ldots$, as $N\to\infty$,
\begin{equation}\label{Equation: H a Approximation}
\begin{split}
&\frac{1}{(2\sqrt{N})^j}\sum_{0\leq i_1\leq\cdots\leq i_j\leq k}\prod_{j'=1}^j\a_{\hat X^{k-j;x}_{N^{-2/3}i_{j'}}}
\approx\frac{1}{j!\,(2\sqrt{N})^j}\bigg(\sum_{i=0}^k \a_{\hat X^{k-j;x}_{N^{-2/3}i}}\bigg)^j \\
&=\frac{1}{j!\,2^j}\bigg(\sum_{a\in N^{-1/3}\N} L^a(X^{k-j;x})\,\frac{\a_{\lfloor N^{1/3}a\rfloor}}{N^{1/6}}\bigg)^j
\to\frac{1}{j!\,2^j}\left(s_\a\int_0^\infty L^a_T(R^{x})\d W^\a_a\right)^j,
\end{split}
\end{equation}
where $W^\a$ is the Brownian motion in a Donsker type invariance principle for the sequence $\a_1,\,\a_2,\,\ldots$ (see \eqref{Equation: B.M.s Joint with Truncated} below). 

\medskip

Finally, the Lebesgue differentiation theorem suggests that
\begin{equation}
N^{1/3}\int_{N^{-1/3}X^{k-j;x}_{T_{k-j}}}^{N^{-1/3}(X^{k-j;x}_{T_{k-j}}+1)}g(y)\d y\to g(R^x_T),\quad N\to\infty. 
\end{equation}

\smallskip

All in all, we expect that, for each $j=0,\,1,\,\ldots$, the quantity $\Sc^j_k(f,g)$ converges, as $N\to\infty$, to
\begin{equation}\label{Equation: H Final Semigroup j-wise}
\begin{split}
\int_0^\infty f(x)\,\E_{R^x}\bigg[\exp\bigg(-\int_0^T\frac{R^x_t}2\d t
+s_\xi\int_0^\infty L_T^a(R^x)\d W^\xi_a-w\frac{L_T^0(R^x)}2\Bigg)\\
\frac{1}{j!}\bigg(\frac{s_\a}{2}\int_0^\infty L^a_T(R^x)\d W^\a_a\bigg)^j \,g(R^x_T)\bigg]\d x.
\end{split}
\end{equation}
Summing over $j=0,\,1,\,\ldots$ and letting $W:=\sqrt{\be}\big(s_\xi W^\xi+\frac{s_\a}{2}W^\a\big)$ we end up precisely with the right-hand side of \eqref{eq:main theorem eq}. 
\end{heuristic}

\subsection{Truncated convergence}\label{Section:truncated}

Our first step in the rigorous proof of Theorem \ref{Theorem: Main} consists in establishing a convergence result for truncated versions of $\Sc^j_k(f,g)$, $j=0,\,1,\,\ldots$. To this end, we define for all $\ubar S\in[-\infty,0]$ and $\bar S\in[0,\infty]$ the truncated functionals
\begin{equation}\label{Equation: Truncated Functional}
\begin{split}
& \widetilde F_j^{(\ubar S,\bar S)}(X^{k-j;x},\a,\xi)=\ubar S\lor \Bigg(\prod_{i=1}^{k-j-H(X^{k-j;x})}\frac{\sqrt{N-\hat X^{k-j;x}_{N^{-2/3}(i-1)}\land\hat X^{k-j;x}_{N^{-2/3}i}}}{\sqrt{N}} \\
&\qquad\qquad\qquad\qquad\cdot\prod_{i=1}^{k-j-H(X^{k-j;x})}\Bigg(1+\frac{\xi_{\hat X^{k-j;x}_{N^{-2/3}(i-1)}\land\hat X^{k-j;x}_{N^{-2/3}i}}}{\sqrt{N-\hat X^{k-j;x}_{N^{-2/3}(i-1)}\land\hat X^{k-j;x}_{N^{-2/3}i}}}\Bigg)\,\ell_N^{H(X^{k-j;x})} \\
&\qquad\qquad\qquad\qquad\cdot\frac{1}{j!\,(2\sqrt{N})^j}\Bigg(\sum_{i=0}^{k-j-H(X^{k-j;x})} \a_{\hat X^{k-j;x}_{N^{-2/3}i}}\Bigg)^j\,\Bigg)\land\bar S,\quad j=0,\,1,\,\ldots
\end{split}
\end{equation}
and for all $K\in\N\cup\{0,\infty\}$ the truncated functions $f_K=f h_K$ and $g_K=g h_K$, where the continuous $h_K:\,[0,\infty)\to[0,1]$ satisfy $h_K\equiv 1$ on $[0,K)$ and $h_K\equiv 0$ on $[2K,\infty)$.  

\begin{remark}
Note that, apart from the truncation at $\ubar S$ and $\bar S$, the functionals $F_j$ and $\widetilde F_j^{(\ubar S,\bar S)}$ differ in the way the $\a_m$'s enter into them.
\end{remark}

We now truncate the terms $\Sc^j_k(f,g)$, $j=0,\,1,\,\ldots$ according to
\begin{equation}
\begin{split}
& \widetilde\Sc^j_k(f,g;\ubar S,\bar S) \\
& :=\int_0^{N^{2/3}} f(x)\,\E_{X^{k-j;x}}\bigg[\widetilde F_j^{(\ubar S,\bar S)}(X^{k-j;x},\a,\xi)\cdot N^{1/3}\int_{N^{-1/3}X^{k-j;x}_{T_{k-j}}}^{N^{-1/3}(X^{k-j;x}_{T_{k-j}}+1)}g(y)\d y\bigg]\d x
\end{split}
\end{equation}
and introduce the limiting operators
\begin{equation}
\begin{split}
\big(\U^{(\ubar S,\bar S)}_{T;j}f\big)(x):=\E_{R^x}\bigg[\bigg(\ubar S\lor\exp\bigg(\!\!-\!\int_0^T \! \frac{R^x_t}2\d t
+s_\xi\int_0^\infty \!\! L_T^a(R^x)\d W^\xi_a-w\frac{L_T^0(R^x)}2\bigg)\\
\cdot\frac{1}{j!}\bigg(\frac{s_\a}{2}\int_0^\infty L^a_T(R^x)\d W^\a_a\bigg)^j\land\bar S\bigg)\,f(R^x_T)\bigg]
\end{split}
\end{equation}
for $f\in{\mathcal D}$ and $j=0,\,1,\,\ldots$. 

\begin{proposition}\label{Proposition: Truncated Convergence}
Let $\ubar S$, $\bar S$ and $K$ be finite. Then, for all functions $f\in{\mathcal D}$ and $g\in {\mathcal D}\cap C([0,\infty))$,
\begin{equation}\label{Equation: Truncated Convergence!}
\lim_{N\to\infty} \widetilde\Sc^j_k(f_K,g_K;\ubar S,\bar S)=\int_0^\infty f_K(x)\,\big(\U^{(\ubar S,\bar S)}_{T;j} g_K\big)(x)\,\mathrm{d}x
\end{equation}
in distribution and in the sense of moments. These convergences hold jointly for any finite collection of $j$'s, $T$'s, $f$'s and $g$'s, and in the case of the convergence in distribution also jointly with the convergences in distribution
\begin{equation}\label{Equation: B.M.s Joint with Truncated}
\lim_{N\to\infty}\sum_{m=0}^{\lfloor N^{1/3}x\rfloor}\frac{\a_m}{N^{1/6}}=s_{\a}W^{\a}_x,\quad x\ge0
\quad\text{and}\quad
\lim_{N\to\infty}\sum_{m=0}^{\lfloor N^{1/3}x\rfloor}\frac{\xi_m}{N^{1/6}}=s_\xi W^{\xi}_x,\quad x\ge0
\end{equation}
with respect to the Skorokhod topology.
\end{proposition}

The key ingredient in the proof of Proposition \ref{Proposition: Truncated Convergence} is the next lemma. Therein and henceforth, for probability measures $\mu$ on $[0,\infty)$, we use the notations $X^{k;\mu}$ and $R^\mu$ for the random walk $X^{k;x}$ started according to the image of $\mu$ under the map $x\mapsto\lfloor xN^{1/3}\rfloor$ and the reflected Brownian motion $R^x$ started according to $\mu$, respectively.

\begin{lemma}\label{Lemma: Basic Convergence}
Let $n\in\N$ and $\mu_1,\,\mu_2,\,\ldots,\,\mu_n$ be probability measures on $[0,\infty)$. Then, there exists a coupling of independent $X^{k;\mu_1},X^{k;\mu_2},\ldots,X^{k;\mu_n}$ with independent $R^{\mu_1},R^{\mu_2},\ldots,R^{\mu_n}$ such that the following limits in distribution hold jointly over $l=1,\,2,\,\ldots,\,n$, and also jointly with \eqref{Equation: B.M.s Joint with Truncated}, 
\begin{eqnarray}
&& \lim_{N\to\infty}\,\sup_{t\in[0,T_k]} \big|N^{-1/3}X^{k;\mu_l}_t-R^{\mu_l}_t\big|=0, 
\label{Equation: Basic Convergence 1} \\
&& \lim_{N\to\infty}\,\sum_{a\in N^{-1/3}(\N-1/2)} L^a(X^{k;\mu_l})\frac{\xi_{\lfloor N^{1/3}a\rfloor}}{N^{1/6}}
=s_\xi\int_0^\infty L^a_T(R^{\mu_l})\d W^\xi_a,
\label{Equation: Basic Convergence 2} \\
&& \lim_{N\to\infty}\,N^{-1/3}H(X^{k;\mu_l})=\frac{L^0_T(R^{\mu_l})}2,
\label{Equation: Basic Convergence 3} \\
&& \lim_{N\to\infty}\,\sum_{a\in N^{-1/3}\N} L^a(X^{k;\mu_l})\frac{\a_{\lfloor N^{1/3}a\rfloor}}{N^{1/6}}
=s_\a\int_0^\infty L^a_T(R^{\mu_l})\d W^\a_a.
\label{Equation: Basic Convergence 4}
\end{eqnarray}
\end{lemma}

\begin{proof} 
The lemma can be obtained from the coupling construction of Theorem \ref{Theorem: Coupling} by the same arguments as in the derivation of \cite[Proposition 4.9]{GS} from the coupling in \cite[Proposition 4.1]{GS}. More specifically, one starts with the case $n=1$ and $\mu_1=\delta_x$ for some $x\ge0$. Then, the joint convergences \eqref{Equation: Basic Convergence 1}-\eqref{Equation: Basic Convergence 4} in distribution are due to the convergence of the associated joint characteristic functions, which under the coupling of Theorem \ref{Theorem: Coupling} is a consequence of the almost sure convergences of the \textit{conditional} characteristic functions 
\begin{equation}\label{Equation: Conditional CF 1}
\begin{split}
& \quad\;\;\; \lim_{N\to\infty}\,
\E_{\xi}\bigg[\exp\bigg(\im\,\theta\,\sum_{a\in N^{-1/3}(\N-1/2)} L^a(X^{k;x})\frac{\xi_{\lfloor N^{1/3}a\rfloor}}{N^{1/6}}\bigg)\bigg]\\  
& \quad\;\;\; =\;\E_{W^\xi}\bigg[\exp\bigg(\im\,\theta\,s_\xi\int_0^\infty L^a_T(R^{x})\d W^\xi_a\bigg)\bigg],
\end{split}
\end{equation}
\begin{equation}\label{Equation: Conditional CF 2}
\begin{split}
& \lim_{N\to\infty}\,\E_{\a}\bigg[\exp\bigg(\im\,\theta\,\sum_{a\in N^{-1/3}\N}L^a(X^{k;x})\frac{\a_{\lfloor N^{1/3}a\rfloor}}{N^{1/6}}\bigg)\bigg] \\
& =\;\E_{W^\a}\bigg[\exp\bigg(\im\,\theta\,s_\a\int_0^\infty L^a_T(R^{x})\d W^\a_a\bigg)\bigg]
\end{split}
\end{equation}
for all $\theta\in\R$ (see \cite[first half of p.~18]{GS} for more details). The latter follow from the central limit theorem in the form of the upper bound in \cite[Theorem 8.4]{BR}, the coupling of Theorem \ref{Theorem: Coupling} and Assumption \ref{Assumption: Matrix Assumption} (see \cite[pp.~18--19]{GS} for more details). 

\medskip

In the case of $n=1$ and a general probability measure $\mu_1$, the joint convergences \eqref{Equation: Basic Convergence 1}-\eqref{Equation: Basic Convergence 4} in distribution can be deduced from the previous case by integrating with respect to $\mu_1$ and relying on the uniform boundedness of characteristic functions. Finally, in the case of $n>1$, one can repeat the same proof, but invoking the multidimensional version of the central limit theorem used before, obtaining this way also the convergences of \eqref{Equation: B.M.s Joint with Truncated} in the sense of convergence of finite-dimensional distributions. The latter can be improved to the desired distributional convergences of processes by applying a standard tightness result (see e.g. \cite[Problem 8.4 and proof of Theorem 8.1]{Bi}). 
\end{proof}

We also prepare the following lemma needed in our proof of Proposition \ref{Proposition: Truncated Convergence}. 

\begin{lemma}\label{Lemma: Basic Convergence 2}
Let $\mu$ be a probability measure on $[0,\infty)$. Then, for each $j=0,\,1,\,\ldots$, under any coupling such that $\lim_{N\to\infty} N^{-1/3}X^{k-j;\mu}_{T_{k-j}}=R^\mu_T$ almost surely, it holds
\begin{equation}
\lim_{N\to\infty}\,N^{1/3}\int_{N^{-1/3}X^{k-j;\mu}_{T_{k-j}}}^{N^{-1/3}(X^{k-j;\mu}_{T_{k-j}}+1)} g(y)\d y=g(R^{\mu}_T)
\end{equation}
with probability one, for any uniformly continuous function $g:\,[0,\infty)\to\R$. 
\end{lemma}

\begin{proof}
It suffices to write
\begin{equation}
\begin{split}
& N^{1/3}\int_{N^{-1/3}X^{k-j;\mu}_{T_{k-j}}}^{N^{-1/3}(X^{k-j;\mu}_{T_{k-j}}+1)}g(y)\d y \\
& =N^{1/3}\int_{N^{-1/3}X^{k-j;\mu}_{T_{k-j}}}^{N^{-1/3}(X^{k-j;\mu}_{T_{k-j}}+1)}g(y)-g\big(N^{-1/3}X^{k-j;\mu}_{T_{k-j}}\big)\d y+g\big(N^{-1/3}X^{k-j;\mu}_{T_{k-j}}\big)
\end{split}
\end{equation}
and to note that the integral on the right-hand side tends to $0$ with probability one, as $N\to\infty$, by the uniform continuity of $g$, whereas $\lim_{N\to\infty} g(N^{-1/3}X^{k-j;\mu}_{T_{k-j}})=g(R^\mu_T)$ almost surely. 
\end{proof}

We are now ready to prove Proposition \ref{Proposition: Truncated Convergence}. 

\begin{proof}[Proof of Proposition \ref{Proposition: Truncated Convergence}]
Let us first consider fixed $j$, $T$, $f$ and $g$. Since the terms $\widetilde\Sc^j_k(f_K,g_K;\ubar S,\bar S)$ are bounded uniformly in $N$, it suffices to show the convergence of moments. Further, without loss of generality we may assume $f_K\geq0$ and $\int_0^\infty f_K(x)\d x=1$ (otherwise we write $f_K$ as the difference of its positive and negative parts, and the latter as multiples of functions of the described kind). In particular, this allows to define $\mu$ as the probability measure with the density $f_K$.

\medskip

With i.i.d. copies $X^{k-j;\mu_1},\,X^{k-j;\mu_2},\,\ldots,\,X^{k-j;\mu_n}$ of $X^{k-j;\mu}$ and i.i.d. copies $R^{\mu_1},\,R^{\mu_2},\,\ldots,\,R^{\mu_n}$ of $R^\mu$, the $n$-th moment of $\widetilde\Sc^j_k(f_K,g_K;\ubar S,\bar S)$ can be expressed using Fubini's theorem as
\begin{equation}\label{Equation: Truncated Moment Prelimit}
\E\bigg[\prod_{l=1}^n \bigg(\widetilde F_j^{(\ubar S,\bar S)}(X^{k-j;\mu_l},\a,\xi)\cdot N^{1/3}\int_{N^{-1/3} X^{k-j;\mu_l}_{T_{k-j}}}^{N^{-1/3}(X^{k-j;\mu_l}_{T_{k-j}}+1)}g_K(y)\d y\bigg)\bigg],
\end{equation}
whereas the $n$-th moment of $\int_0^\infty f_K(x)\,\big(\U^{(\ubar S,\bar S)}_{T;j} g_K\big)(x)\,\mathrm{d}x$ reads
\begin{equation}\label{Equation: Truncated Moment Limit}
\begin{split}
\E\bigg[\prod_{l=1}^n \bigg(\bigg(\ubar S\lor\exp\bigg(-\int_0^T\frac{R^{\mu_l}_t}2\d t
+s_\xi\int_0^\infty L_T^a(R^{\mu_l})\d W^\xi_a-w\frac{L_T^0(R^{\mu_l})}2\bigg)\\
\cdot \frac{1}{j!}\bigg(\frac{s_\a}{2}\int_0^\infty L_T^a(R^{\mu_l})\d W^\a_a\bigg)^j \land\bar S\bigg)\,g_K(R^{\mu_l}_T)\bigg)\bigg].
\end{split}
\end{equation}
To establish the convergence of the expectation in \eqref{Equation: Truncated Moment Prelimit} to that in \eqref{Equation: Truncated Moment Limit} we work under the coupling of Lemma \ref{Lemma: Basic Convergence} and view the random walks $X^{k-j;\mu_1},\,X^{k-j;\mu_2},\,\ldots,$ $X^{k-j;\mu_n}$ as the respective restrictions of $X^{k;\mu_1},\,X^{k;\mu_2},\,\ldots,\,X^{k;\mu_n}$ to $[0,T_{k-j}]$. Then, $X^{k-j;\mu_1},\,X^{k-j;\mu_2},\,\ldots,$ $X^{k-j;\mu_n}$ inherit the asymptotics \eqref{Equation: Basic Convergence 1}-\eqref{Equation: Basic Convergence 4} from $X^{k;\mu_1}$, $X^{k;\mu_2},\,\ldots,\,X^{k;\mu_n}$, and Lemma \ref{Lemma: Basic Convergence 2} applies, so that  
\begin{equation} \label{Equation: Truncated Functional Convergence 2}
\lim_{N\to\infty}\,N^{1/3}\int_{N^{-1/3} X^{k-j;\mu_l}_{T_{k-j}}}^{N^{-1/3}(X^{k-j;\mu_l}_{T_{k-j}}+1)} g_K(y)\d y=g_K(R^{\mu_l}_T),\quad l=1,\,2,\,\ldots,\,n
\end{equation}
with probability one. 

\smallskip

We proceed to the asymptotics of $\widetilde F_j^{(\ubar S,\bar S)}(X^{k-j;\mu_l},\a,\xi)$, $l=1,\,2,\,\ldots,\,n$. Our first claim is that 
\begin{equation}\label{eq:limit area}
\lim_{N\to\infty}\prod_{i=1}^{k-j-H(X^{k-j;\mu_l})}
\frac{\sqrt{N-\hat X^{k-j;\mu_l}_{N^{-2/3}(i-1)}\land\hat X^{k-j;\mu_l}_{N^{-2/3}i}}}{\sqrt{N}}
=\exp\bigg(-\int_0^T\frac{R^{\mu_l}_t}2\d t\bigg)
\end{equation}
for $l=1,\,2,\,\ldots,\,n$ almost surely. Indeed, for every such $l$, according to the Taylor expansion $\log(1+z)=z+O(z^2)$ about $z=0$, an approximation as in the third line of \eqref{Equation: H Area Approximation} (with $X^{k-j;x}$, $\hat X^{k-j;x}$ replaced by $X^{k-j;\mu_l}$, $\hat X^{k-j;\mu_l}$) holds up to a multiplicative error of at most
\begin{equation}
\exp\bigg(O\bigg(N^{-4/3}\Big(\sup_{t\in[0,T_{k-j}]} X^{k-j;\mu_l}_t\Big)^2\bigg)\bigg).
\end{equation}
Writing the resulting approximation in terms of $X^{k-j;\mu_l}$ we obtain \eqref{eq:limit area} as an elementary consequence of \eqref{Equation: Basic Convergence 1}.

\medskip

Next, we prove the joint convergence in distribution
\begin{equation}\label{eq:xi_part_conv}
\begin{split}
& \lim_{N\to\infty} \prod_{i=1}^{k-j-H(X^{k-j;\mu_l})}
\left(1+\frac{\xi_{\hat X^{k-j;\mu_l}_{N^{-2/3}(i-1)}\land \hat X^{k-j;\mu_l}_{N^{-2/3}i}}}{\sqrt{N-\hat X^{k-j;\mu_l}_{N^{-2/3}(i-1)}\land \hat X^{k-j;\mu_l}_{N^{-2/3}i}}}\right) \\
&=\exp\bigg(s_\xi\int_0^\infty L^a_T(R^{\mu_l})\d W^\xi_a\bigg)
\end{split}
\end{equation}
for $l=1,\,2,\,\ldots,\,n$. To this end, we use the Taylor expansion $(1-z)^{-1/2}=1+O(z)$ about $z=0$ to conclude that, for each $l$, an approximation as in the third line of \eqref{Equation: H xi Approximation 0} (with $X^{k-j;\mu_l}$, $\hat X^{k-j;\mu_l}$ in place of $X^{k-j;x}$, $\hat X^{k-j;x}$) applies up to a modification of each
\begin{equation}
\frac{\xi_{\hat X^{k-j;\mu_l}_{N^{-2/3}(i-1)}\land \hat X^{k-j;\mu_l}_{N^{-2/3}i}}}{\sqrt{N}}\;\;\;\text{to}\;\;\;
\frac{\xi_{\hat X^{k-j;\mu_l}_{N^{-2/3}(i-1)}\land \hat X^{k-j;\mu_l}_{N^{-2/3}i}}}{\sqrt{N}}\bigg(1+O\Big(N^{-1}\sup_{t\in[0,T_{k-j}]} X^{k-j;\mu_l}_t \Big)\bigg).
\end{equation}
At this point, we employ the Taylor expansion $\log(1+z)=z+O(z^2)$ about $z=0$ to obtain an expression as in the fourth line of \eqref{Equation: H xi Approximation 0}, with the summands therein modified to 
\begin{equation}\label{eq:xi_summand_expansion}
\begin{split}
& \frac{\xi_{\hat X^{k-j;\mu_l}_{N^{-2/3}(i-1)}\land \hat X^{k-j;\mu_l}_{N^{-2/3}i}}}{\sqrt{N}}\bigg(1+O\Big(N^{-1}\sup_{t\in[0,T_{k-j}]} X^{k-j;\mu_l}_t \Big)\bigg) \\
& +O\Bigg(\frac{\big(\xi_{\hat X^{k-j;\mu_l}_{N^{-2/3}(i-1)}\land \hat X^{k-j;\mu_l}_{N^{-2/3}i}}\big)^2}{N}\bigg(1+O\Big(N^{-1}\sup_{t\in[0,T_{k-j}]} X^{k-j;\mu_l}_t \Big)\bigg)^2\Bigg).
\end{split}
\end{equation}
The contribution of the first line in \eqref{eq:xi_summand_expansion} can be evaluated as in the equality on the fifth line of \eqref{Equation: H xi Approximation 0}, which leads to the limit in distribution of \eqref{Equation: Basic Convergence 2} after recalling \eqref{Equation: Basic Convergence 1}. The contribution of the second line in \eqref{eq:xi_summand_expansion} is asymptotically negligible due to the almost sure convergence 
\begin{equation}\label{eq:conv_in_prob_0}
\lim_{N\to\infty} \sum_{i=1}^{k-j-H(X^{k-j;\mu_l})} \;\frac{\big(\xi_{\hat X^{k-j;\mu_l}_{N^{-2/3}(i-1)}\land \hat X^{k-j;\mu_l}_{N^{-2/3}i}}\big)^2}{N}=0
\end{equation}
(simply apply the Borel-Cantelli lemma upon bounding the fourth moment of the latter sum via Assumption \ref{Assumption: Matrix Assumption}(c)) and \eqref{Equation: Basic Convergence 1}. All in all, we arrive at \eqref{eq:xi_part_conv}.

\medskip

Putting \eqref{eq:limit area} and \eqref{eq:xi_part_conv} together with  the almost sure convergences 
\begin{equation}
\ell_N^{H(X^{k-j;\mu_l})}=\bigg(1-\frac{N^{1/3}(1-\ell_N)}{N^{1/3}}\bigg)^{N^{1/3}\cdot(H(X^{k-j;\mu_l})/N^{1/3})}\to\exp\bigg(-w\frac{L_T^0(R^{\mu_l})}2\bigg) 
\end{equation}
for $l=1,\,2,\,\ldots,\,n$ (see \eqref{Equation: Basic Convergence 3}), the convergences in distribution
\begin{equation} 
\lim_{N\to\infty}\frac{1}{j!\,(2\sqrt{N})^j}\Bigg(\sum_{i=0}^{k-j-H(X^{k-j;\mu_l})} \a_{\hat X^{k-j;\mu_l}_{N^{-2/3}i}}\Bigg)^j
=\frac{1}{j!}\bigg(\frac{s_\a}{2}\int_0^\infty L_T^a(R^{\mu_l})\d W^\a_a\bigg)^j
\end{equation}
for $l=1,\,2,\,\ldots,\,n$ (see \eqref{Equation: Basic Convergence 4})  and \eqref{Equation: Truncated Functional Convergence 2} we conclude that the expectation in \eqref{Equation: Truncated Moment Prelimit} converges to that in \eqref{Equation: Truncated Moment Limit}. Moreover, the joint convergence for any finitely many $j$'s, $T$'s, $f$'s and $g$'s can be shown by the same arguments, the only difference being that the formulas for moments in \eqref{Equation: Truncated Moment Prelimit} and \eqref{Equation: Truncated Moment Limit} have to be replaced by the corresponding formulas for joint moments. 
\end{proof}

\subsection{Uniform moment bounds}\label{Section:uniform_moment_bounds}

In this subsection, we establish some uniform moment estimates, which will allow us to lift the truncations and the continuity assumption on the $g$'s of Proposition \ref{Proposition: Truncated Convergence}. To this end, we define the functionals
\begin{equation}
\begin{split}
& \widetilde F_j(X^{k-j;x},\a,\xi)=\prod_{i=1}^{k-j-H(X^{k-j;x})}
\frac{\sqrt{N-\hat X^{k-j;x}_{N^{-2/3}(i-1)}\land\hat X^{k-j;x}_{N^{-2/3}i}}}{\sqrt{N}} \\
&\qquad\qquad\qquad\quad\;\;\; \cdot\prod_{i=1}^{k-j-H(X^{k-j;x})} \left(1+\frac{\xi_{\hat X^{k-j;x}_{N^{-2/3}(i-1)}\land\hat X^{k-j;x}_{N^{-2/3}i}}}{\sqrt{N-\hat X^{k-j;x}_{N^{-2/3}(i-1)}\land\hat X^{k-j;x}_{N^{-2/3}i}}}\right)
\ell_N^{H(X^{k-j;x})} \\
&\qquad\qquad\qquad\quad\;\;\; \cdot\frac{1}{j!\,(2\sqrt{N})^j}\Bigg(\sum_{i=0}^{k-j-H(X^{k-j;x})} \a_{\hat X^{k-j;x}_{N^{-2/3}i}}\Bigg)^j,\quad j=0,\,1,\,\ldots
\end{split}
\end{equation}
and, with any $f,g\in{\mathcal D}$ and random variable $Z_N$ (possibly depending on $X^{k-j;x}$, the $\a_m$'s and the $\xi_m$'s), set
\begin{equation}
\begin{split}
& \widetilde\Sc^j_k(f,g;Z_N) \\
& :=\int_0^{N^{2/3}} \!\!\!\! f(x)\,\E_{X^{k-j;x}}\bigg[\widetilde F_j(X^{k-j;x},\a,\xi)\cdot N^{1/3}\int_{N^{-1/3}X^{k-j;x}_{T_{k-j}}}^{N^{-1/3}(X^{k-j;x}_{T_{k-j}}+1)}g(y)\d y\cdot Z_N\bigg]\!\!\d x
\end{split}
\end{equation}
for $j=0,\,1,\,\ldots$. We also let 
\begin{equation}
\bar\Sc^j_k(f,g,K;Z_N):=\widetilde\Sc^j_k(f,g;Z_N)
-\widetilde\Sc^j_k(f_K,g_K;Z_N),\quad j=0,\,1,\,\ldots
\end{equation}
for $K\in\N\cup\{0\}$.

\begin{proposition}\label{Proposition: Uniform Moment Bounds}
For all $f,g\in{\mathcal D}$, one can find $N_0\in\N$ and $C(K,n)<\infty$ with $\lim_{K\to\infty}C(K,n)=0$ such that, for all $N\ge N_0$, if $Z_N$ satisfies
\begin{equation}\label{eq:ZNmoments}
\E\big[|Z_N|^{3n}\big]\le \Theta(3n),
\end{equation}
one has
\begin{equation}\label{Equation: Uniform Moment Bounds}
\E\Big[\big|\bar\Sc^j_k(f,g,K;Z_N)\big|^n\Big]
\leq\frac{C(K,n)\,\Theta(3n)^{1/3}}{(3/2)^{jn}},\quad K\in\N\cup\{0\},\;\;j=0,\,1,\,\ldots.
\end{equation}
\end{proposition}

The proof of Proposition \ref{Proposition: Uniform Moment Bounds} relies on the following lemma. 

\begin{lemma}\label{Lemma: Large Deviations}
For all $1\leq p<3$ and $\theta\ge0$, there exist constants $C=C(p,\theta)<\infty$, $c=c(\theta)>0$ and $N_0\in\N$ such that  
\begin{eqnarray}
\label{Equation: Uniform Moment Estimate 1}
&& \sup_{N\ge N_0}\,\E\bigg[\exp\bigg(-\theta N^{-2/3}\sum_{i=1}^k N^{-1/3} X^{k;x}_{N^{-2/3}i}\bigg)\bigg] \le Ce^{-cx},\quad x>0, \\
\label{Equation: Uniform Moment Estimate 2}
&& \sup_{N\ge N_0}\,\E\bigg[\exp\bigg(\theta N^{-1/3}\sum_{a\in N^{-1/3}\N}L^a(X^{k;x})^p\bigg)\bigg]\le C,\quad x>0, \\
\label{Equation: Uniform Moment Estimate 2.1}
&& \sup_{N\ge N_0}\,\E\bigg[\exp\bigg(\theta N^{-1/3}\sum_{a\in N^{-1/3}(\N-1/2)} L^a(X^{k;x})^p\bigg)\bigg]\le C,\quad x>0, \\
\label{Equation: Uniform Moment Estimate 3}
&& \sup_{N\ge N_0}\,\E\Big[\ell_N^{\theta H(X^{x;k})}\Big]\le C,\quad x>0.
\end{eqnarray}
\end{lemma}

\begin{proof}
Recall the random walks $\bar Y$ and $v(\bar Y)$ introduced in \eqref{eq:whatisYbar} and the sentence following it. Throughout this proof, we condition on $\bar Y_0=\lfloor N^{1/3} x\rfloor$ and assume that $\bar Y$ and $X^{k;x}$ are coupled as in Remark \ref{Remark: Lazy Path 2}. We further write
\begin{equation}
\rho_k=\max_{0\le i\le k} v(\bar Y)_i - \min_{0\le i\le k} v(\bar Y)_i
\end{equation}
for the range of the SSRW $v(\bar Y)$ after $k$ steps. It is clear that, for $i=0,\,1,\,\ldots,\,k$,
\begin{equation}\label{eq: bound by range}
-X^{k;x}_{N^{-2/3}i}\leq-\min_{0\le i\le k} \bar Y_i
\leq -\min_{0\le i\le k} v(\bar Y)_i
\leq -\lfloor N^{1/3} x\rfloor+\rho_k.
\end{equation}
According to \cite[inequality (6.2.3)]{Ch} (the case $p=1$ therein), one has
\begin{equation}\label{Equation: RW Range Moments}
\E\Big[\big(N^{-1/3}\rho_k\big)^n\Big]\leq \sqrt{n!}\,\big(CT^{1/2}\big)^n,\quad n\in\N
\end{equation}
with some uniform constant $C<\infty$. Thus, the exponential moment of $N^{-1/3}\rho_k$ can be bounded by a constant independent of $N$ and $x$, yielding \eqref{Equation: Uniform Moment Estimate 1}.

\medskip

In view of \eqref{Equation: Lazy Local Time},
\begin{equation}
L^a(X^{k;x})^p\leq 2^{p-1}\Big(L^a_{N^{2/3};T}\big(v(\bar Y)\big)^p+L^{-a}_{N^{2/3};T}\big(v(\bar Y)\big)^p\Big).
\end{equation}
Hence, it suffices to show \eqref{Equation: Uniform Moment Estimate 2} with $\theta$ replaced by $2^{p-1}\theta$ and $\sum_{a\in N^{-1/3}\N}L^a(X^{k;x})^p$ by $\sum_{a\in N^{-1/3}\Z\setminus\{0\}}L^a_{N^{2/3};T}\big(v(\bar Y)\big)^p$. Repeating the proof of \cite[Proposition 4.3]{GS} verbatim we find a constant $C=C(p)<\infty$ such that 
\begin{equation}
N^{-1/3}\sum_{a\in N^{-1/3}\Z\setminus\{0\}}L^a_{N^{2/3};T}\big(v(\bar Y)\big)^p \leq C\Big(\big(N^{-1/3}\rho_k\big)^{p-1}+N^{-(p-1)/3}\Big)
\end{equation}
(note that even though \cite{GS} considers SSRW bridges, all the combinatorial identities therein apply to SSRWs as well). At this point, \eqref{Equation: Uniform Moment Estimate 2} follows from
\eqref{Equation: RW Range Moments}. Moreover, \eqref{Equation: Uniform Moment Estimate 2.1} is a consequence of  
\begin{equation}
L^a(X^{k;x})\le 2L^{a+N^{-1/3}/2}(X^{k;x}),\quad a\in N^{-1/3}(\N-1/2)
\end{equation}
and \eqref{Equation: Uniform Moment Estimate 2}.

\medskip

From Proposition \ref{Proposition: Discrete Skorokhod} we know that $H(X^{x;k})=\sup_{t\in[0,T_k]} (-\widetilde X^{k;x}_t)_+$ under an appropriate coupling. In particular, $H(X^{x;k})$ is stochastically dominated by $\rho_k$. Consequently, for $N\in\N$ large enough, the random variable inside the expectation in \eqref{Equation: Uniform Moment Estimate 3} is stochastically dominated by
\begin{equation}
\ell_N^{\theta \rho_k}
=\bigg(1-\frac{N^{1/3}(1-\ell_N)}{N^{1/3}}\bigg)^{N^{1/3}\cdot(\theta N^{-1/3}\rho_k)}
=e^{-(w+o(1))\cdot(\theta N^{-1/3}\rho_k)},
\end{equation}
where $o(1)$ is non-random, so that \eqref{Equation: Uniform Moment Estimate 3} also follows from \eqref{Equation: RW Range Moments}.
\end{proof}

\begin{proof}[Proof of Proposition \ref{Proposition: Uniform Moment Bounds}]
After observing
\begin{equation}\label{eq:symmetry}
\begin{split}
& \bar\Sc^j_k(f,g,K;Z_N)=\int_0^{N^{2/3}}\int_0^{N^{2/3}}N^{1/3}f(x) \\
&\qquad\qquad\quad\;\cdot\frac{Q_{k-j}^{x,y}}{2^{k-j}}\,
\E_{X^{k-j;x}}\big[\widetilde F_j(X^{k-j;x},\a,\xi)\cdot Z_N\,\big|\,X^{k-j;x}_{T_{k-j}}=\lfloor N^{1/3}y\rfloor\big]\,g(y)\d x\d y \\
&\qquad\qquad\qquad\qquad -\int_0^{N^{2/3}}\int_0^{N^{2/3}} N^{1/3}f(x)\,h_K(x) \\
&\qquad\;\; \cdot\frac{Q_{k-j}^{x,y}}{2^{k-j}}\,\E_{X^{k-j;x}}\big[\widetilde F_j(X^{k-j;x},\a,\xi)\cdot Z_N\,\big|\,X^{k-j;x}_{T_{k-j}}=\lfloor N^{1/3}y\rfloor\big]\,g(y)\,h_K(y)\d x\d y
\end{split}
\end{equation}
we estimate $\big|\bar\Sc^j_k(f,g,K;Z_N)\big|$ by moving the absolute value inside the double integral and using
\begin{equation}
0\le 1-h_K(x)\,h_K(y)\le \mathbf{1}_{[K,\infty)}(x)+\mathbf{1}_{[K,\infty)}(y),\quad x,y\ge0.  
\end{equation}
Since the roles of the variables $x$ and $y$ are symmetric, we only focus on the term in $\E\big[\big|\bar\Sc^j_k(f,g,K;Z_N)\big|^n\big]$ originating from $\mathbf{1}_{[K,\infty)}(x)$. We bound the latter by inserting an absolute value into the conditional expectation, applying Fubini's theorem and letting $\widetilde{f}_K:=f\,\mathbf{1}_{[K,\infty)}$, thereby obtaining
\begin{equation}\label{eq:Fubini again}
\begin{split}
& \int_{[0,N^{2/3}]^n}\!\! \E\Bigg[\prod_{l=1}^n \big|\widetilde{f}_K(x_l)\big|\,\E_{X^{k-j;x_l}}\bigg[\big|\widetilde F_j(X^{k-j;x_l},\a,\xi)\big| \\
&\qquad\qquad\qquad\qquad\quad\;\;
\cdot N^{1/3}\int_{N^{-1/3}X^{k-j;x_l}_{T_{k-j}}}^{N^{-1/3}(X^{k-j;x_l}_{T_{k-j}}+1)} |g(y)|\d y\cdot |Z_N|\bigg]\Bigg]
\d x_1 \d x_2 \ldots \d x_n.
\end{split}
\end{equation}
A repeated application of H\"older's and Jensen's inequalities shows further that the quantity in \eqref{eq:Fubini again} is at most
\begin{equation}\label{Equation: Product of Integrals}
\begin{split}
\Bigg(\int_0^{N^{2/3}} \big|\widetilde f_K(x)\big|\,\E\big[|Z_N|^{3n}\big]^{1/(3n)}\,\E\Bigg[\Bigg(N^{1/3}\int_{N^{-1/3}X^{k-j;x}_{T_{k-j}}}^{N^{-1/3}(X^{k-j;x}_{T_{k-j}}+1)} |g(y)|\d y\Bigg)^{\!\!3n}\Bigg]^{1/(3n)} 
\;\;\; \\
\cdot\E\Big[\big| \widetilde F_j(X^{k-j;x},\a,\xi)\big|^{3n}\Big]^{1/(3n)}\d x\Bigg)^n.
\end{split}
\end{equation}

\smallskip

Due to $f\in{\mathcal D}$ and \eqref{eq:ZNmoments}, we have
\begin{equation}\label{Equation: f and Z_N}
\big|\widetilde f_K(x)\big|\,\E\big[|Z_N|^{3n}\big]^{1/(3n)}
\leq C_1 e^{C_2 x^{1-\delta}}\,\mathbf{1}_{(K,\infty)}(x)\,\Theta(3n)^{1/3n}.
\end{equation}
In view of $g\in{\mathcal D}$, we can choose $C_1$, $C_2$ and $\delta$ such that also
\begin{equation}
N^{1/3}\int_{N^{-1/3}X^{k-j;x}_{T_{k-j}}}^{N^{-1/3}(X^{k-j;x}_{T_{k-j}}+1)}|g(y)|\d y\leq C_1 \exp\Big(C_2\big(N^{-1/3}X^{k-j;x}_{T_{k-j}})^{1-\delta}\Big).
\end{equation}
Moreover, by the argument leading to \eqref{eq: bound by range} and with the same notation as there,
\begin{equation}
N^{-1/3}X^{k;x}_{T_{k-j}}\leq x+N^{-1/3}\rho_k.
\end{equation}
It then follows from \eqref{Equation: RW Range Moments} that, with some $\widetilde{C}_1=\widetilde{C}_1(C_1,n)<\infty$,
\begin{equation}\label{Equation: Lebesgue Differentiation Bound}
\E\Bigg[\Bigg(N^{1/3}\int_{N^{-1/3}X^{k-j;x}_{T_{k-j}}}^{N^{-1/3}(X^{k-j;x}_{T_{k-j}}+1)}|g(y)|\d y\Bigg)^{\!\!3n}\Bigg]^{1/(3n)}\leq \widetilde{C}_1 e^{C_2x^{1-\delta}}.
\end{equation}

\smallskip

It remains to control $\E\big[|\widetilde F_j(X^{k-j;x},\a,\xi)|^{3n}\big]^{1/(3n)}$. For this purpose, we fix an $\eps\in(0,\delta/3)$ and distinguish the cases $x\in(0,N^{2/3-\eps}]$ and $x\in(N^{2/3-\eps},N^{2/3}]$. In the first case, we use H\"older's inequality to estimate $\E\big[|\widetilde F_j(X^{k-j;x},\a,\xi)|^{3n}\big]^{1/(3n)}$ by the product of the four terms
\begin{eqnarray}
&& \E\Bigg[\Bigg(\prod_{i=1}^{k-j-H(X^{k-j;x})}
\frac{\sqrt{N-\hat X^{k-j;x}_{N^{-2/3}(i-1)}\land\hat X^{k-j;x}_{N^{-2/3}i}}}{\sqrt{N}}\,\Bigg)^{\!\!12n}\Bigg]^{1/(12n)},
\label{Equation: pre Functional Moment Bound 1} \\
&& \E\left[\!\left(\prod_{i=1}^{k-j-H(X^{k-j;x})} \!\left|1\!+\!\frac{\xi_{\hat X^{k-j;x}_{N^{-2/3}(i-1)}\land\hat X^{k-j;x}_{N^{-2/3}i}}}{\sqrt{N-\hat X^{k-j;x}_{N^{-2/3}(i-1)}\land\hat X^{k-j;x}_{N^{-2/3}i}}}\right|\;\right)^{\!\!12n}\right]^{1/(12n)},
\label{Equation: pre Functional Moment Bound 2} \\
&& \E\Big[\ell_N^{12nH(X^{k-j;x})}\Big]^{1/(12n)}, 
\label{Equation: pre Functional Moment Bound 3} \\
&& \E\Bigg[\Bigg(\frac{1}{j!\,(2\sqrt{N})^j}\Bigg|\sum_{i=0}^{k-j-H(X^{k-j;x})} \a_{\hat X^{k-j;x}_{N^{-2/3}i}}\Bigg|^j\;\Bigg)^{\!\! 12n}\Bigg]^{1/(12n)}.
\label{Equation: pre Functional Moment Bound 4}
\end{eqnarray}
Thanks to $\frac{\sqrt{N-z}}{\sqrt{N}}\le e^{-z/(2N)}$, $z\in[0,N]$ and \eqref{Equation: Uniform Moment Estimate 1}, the quantity in \eqref{Equation: pre Functional Moment Bound 1} is not greater than
$Ce^{-cx+O(jN^{-\eps})}$. Turning to the term in \eqref{Equation: pre Functional Moment Bound 2}, we write the expectation with respect to the $\xi_m$'s as a product and note that, due to Assumption \ref{Assumption: Matrix Assumption}(c), \cite[inequality (4.21)]{GS} yields for each factor a bound of the form 
\begin{equation}
\exp\Bigg(\frac{12n L^a(X^{k-j;x})\big|\E[\xi_a]\big|}{N^{-1/3}\sqrt{N-a}}+C'\Bigg(\frac{\big(12n L^a(X^{k-j;x})\big)^2}{N^{-2/3}(N-a)} +\frac{\big(12n L^a(X^{k-j;x})\big)^{\gamma'}}{N^{-\gamma'/3}(N-a)^{\gamma'/2}}\Bigg)\!\!\Bigg),
\end{equation}
with some $C'<\infty$ and $2<\gamma'<3$. For $N\in\N$ large enough, $N-a\ge N/2$ when $L^a(X^{k-j;x})\neq 0$, which with Assumption \ref{Assumption: Matrix Assumption}(a) leads to the expectation of
\begin{equation}
\exp\!\Bigg(C\sum_{a\in N^{-1/3}(\N-1/2)}\!\frac{12nL^a(X^{k-j;x})}{N^{1/2}}+\frac{\big(12n L^a(X^{k-j;x})\big)^2}{N^{1/3}}+\frac{\big(12n L^a(X^{k-j;x})\big)^{\gamma'}}{N^{\gamma'/6}}\Bigg)
\end{equation}
as an estimate on the expectation in \eqref{Equation: pre Functional Moment Bound 2}. In addition, \eqref{Equation: Uniform Moment Estimate 3} reveals that the expression in \eqref{Equation: pre Functional Moment Bound 3} is at most $Ce^{O(jN^{-1/3})}$. Finally, the expectation with respect to the $\a_m$'s in \eqref{Equation: pre Functional Moment Bound 4} can be controlled via a combination of $\frac{|z|^j}{j!}\le e^{|z|}\le e^z+e^{-z}$, $z\in\R$, \cite[inequality (4.20)]{GS} and Assumption \ref{Assumption: Matrix Assumption}(a) by
\begin{equation}
\frac{C}{2^{12jn}}\exp\!\Bigg(C\sum_{a\in N^{-1/3}\N}\! \frac{12nL^a(X^{k-j;x})}{N^{1/2}}+\frac{\big(12nL^a(X^{k-j;x})\big)^2}{N^{1/3}}+\frac{\big(12nL^a(X^{k-j;x})\big)^{\gamma'}}{N^{\gamma'/6}}\Bigg),
\end{equation}
with the same $2<\gamma'<3$ as before. Putting everything together, applying H\"older's inequality, and appealing to \eqref{Equation: Uniform Moment Estimate 2.1}, \eqref{Equation: Uniform Moment Estimate 2} we arrive at
\begin{equation}\label{Equation: Functional Moment Bound 1}
\E\Big[\big|\widetilde F_j(X^{k-j;x},\a,\xi)\big|^{3n}\Big]^{1/(3n)}
\leq\frac{Ce^{-cx+O(jN^{-\eps})}}{2^j},\quad x\in\big(0,N^{2/3-\eps}\big].
\end{equation}

\smallskip

In the case $x\in(N^{2/3-\eps},N^{2/3}]$, for all $N\in\N$ large enough, $X^{k-j;x}_t\ge N^{1-\eps}/2$, $t\in[0,T_{k-j}]$ , so that  
\begin{equation}
\begin{split}
&\prod_{i=1}^{k-j-H(X^{k-j;x})}
\left|\frac{\sqrt{N-\hat X^{k-j;x}_{N^{-2/3}(i-1)}\land \hat X^{k-j;x}_{N^{-2/3}i}}+\xi_{\hat X^{k-j;x}_{N^{-2/3}(i-1)}\land \hat X^{k-j;x}_{N^{-2/3}i}}}{\sqrt{N}}\right| \\
&\leq\prod_{i=1}^{k-j}
\left(\frac{\sqrt{N-N^{1-\eps}/2}+\big|\xi_{\hat X^{k-j;x}_{N^{-2/3}(i-1)}\land \hat X^{k-j;x}_{N^{-2/3}i}}\big|}{\sqrt{N}}\right) \\
&=\bigg(1-\frac{1}{2N^\eps}\bigg)^{(k-j)/2}\cdot
\prod_{i=1}^{k-j} \left(1+\frac{\big|\xi_{\hat X^{k-j;x}_{N^{-2/3}(i-1)}\land \hat X^{k-j;x}_{N^{-2/3}i}}\big|}
{\sqrt{N-N^{1-\eps}/2}}\right) \\
&\leq e^{-kN^{-\eps}/4+O(jN^{-\eps})}\cdot\exp\bigg(2\sum_{a\in N^{-1/3}(\N-1/2)} L^a(X^{k-j;x})\frac{|\xi_{\lfloor N^{1/3}a\rfloor}|}{N^{1/6}}\bigg).
\end{split}
\end{equation}
Using H\"older's inequality and Lemma \ref{Lemma: Large Deviations} as above, but changing the $\a_m$'s and $\xi_m$'s to their absolute values and Assumption \ref{Assumption: Matrix Assumption}(a) to Assumption \ref{Assumption: Matrix Assumption}(c), we get
\begin{equation}\label{Equation: Functional Moment Bound 2}
\E\Big[\big|\widetilde F_j(X^{k-j;x},\a,\xi)\big|^{3n}\Big]^{1/(3n)}
\leq \frac{Ce^{-kN^{-\eps}/4+O(jN^{-\eps})+CN^{1/6}}}{2^j},\;\;\; x\in\big(N^{2/3-\eps},N^{2/3}\big].
\end{equation}

\smallskip

Lastly, we insert the right-hand sides of \eqref{Equation: f and Z_N}, \eqref{Equation: Lebesgue Differentiation Bound}, \eqref{Equation: Functional Moment Bound 1}, \eqref{Equation: Functional Moment Bound 2} into \eqref{Equation: Product of Integrals}: 
\begin{equation}\label{Equation: Moment Bound Final}
\begin{split}
& \frac{Ce^{O(jnN^{-\eps})}\,\Theta(3n)^{1/3}}{2^{jn}} \\
&\cdot \Bigg(\int_K^{N^{2/3-\eps}\lor K} e^{Cx^{1-\delta}-cx}\d x
+e^{-kN^{-\eps}/4+CN^{1/6}}\int_{N^{2/3-\eps}\lor K}^{N^{2/3}\lor K} e^{Cx^{1-\delta}}\d x\Bigg)^n.
\end{split}
\end{equation}
The estimate \eqref{Equation: Uniform Moment Bounds} readily follows upon recalling $k=\lfloor TN^{2/3}\rfloor$ and $\eps\in(0,\delta/3)$. 
\end{proof}

\subsection{Proof of Theorem \ref{Theorem: Main}}\label{Section: Final Proof Synthesis}

In order to establish Theorem \ref{Theorem: Main}, we first argue that Proposition \ref{Proposition: Truncated Convergence} remains true when $\ubar S=-\infty$ and $\bar S=\infty$. Indeed, the convergence in distribution of the random variable inside the expectation of \eqref{Equation: Truncated Moment Prelimit} to that inside the expectation of \eqref{Equation: Truncated Moment Limit} continues to hold. This yields the convergence \eqref{Equation: Truncated Convergence!} with $\ubar S=-\infty$ and $\bar S=\infty$ in the sense of moments, since the random variable of \eqref{Equation: Truncated Moment Prelimit} is uniformly integrable, as $N\in\N$ varies, by \eqref{Equation: Uniform Moment Bounds} with $K=0$ and $Z_N=1$. The same argument gives also the joint convergence in the sense of moments. The convergence \eqref{Equation: Truncated Convergence!} with $\ubar S=-\infty$ and $\bar S=\infty$ in distribution results from an identification of the limit points in distribution with the limit in the sense of moments via the stochastic domination argument in \cite[proof of Lemma 4.15]{GS}. Such identification in the case of the joint convergence in distribution is most easily seen after an application of the Skorokhod representation theorem, leading to two random vectors with componentwise inequalities between them and same joint moments. 

\medskip

Next, we prove Theorem \ref{Theorem: Main} under the additional assumption that the $g$'s therein are continuous. Let 
\begin{equation}
\Delta^j_k(f,g):=\Sc^j_k(f,g)-\widetilde\Sc^j_k(f,g;1),\quad N\in\N,\;\;j=0,\,1,\,\ldots.
\end{equation}
In view of \eqref{Equation: Weak Convergence Prelimit}, for all $K\in\N$, one has
\begin{equation}\label{Equation: Weak Convergence Prelimit Final}
\big(\pi_Nf\big)^\top\M^{\be;w}_{T;N}\big(\pi_Ng\big)=\sum_{j=0}^k\widetilde\Sc^j_k(f_K,g_K;1)+\sum_{j=0}^k\bar\Sc^j_k(f,g,K;1) 
+\sum_{j=0}^k \Delta^j_k(f,g).
\end{equation}

\smallskip

We aim to take the $N\to\infty$ limit of the right-hand side in \eqref{Equation: Weak Convergence Prelimit Final} and start with the asymptotics of the first sum therein. For every finite set of summands $\widetilde\Sc^{1}_k(f_K,g_K;1)$, $\widetilde\Sc^{2}_k(f_K,g_K;1),\;\ldots,\;\widetilde\Sc^{J}_k(f_K,g_K;1)$, their joint limit in distribution and in the sense of moments is determined by the right-hand side of \eqref{Equation: Truncated Convergence!} with $\ubar S=-\infty$ and $\bar S=\infty$. This and the moment bounds of \eqref{Equation: Uniform Moment Bounds} imply that the first sum on the right-hand side of \eqref{Equation: Weak Convergence Prelimit Final} converges to $\int_0^\infty f_K(x)\,\big(\U^{\be;w}_Tg_K\big)(x)\d x$ in distribution and in the sense of moments. Since the moment bounds of \eqref{Equation: Uniform Moment Bounds} for $\bar\Sc^j_k(f_K,g_K,0;1)$ are inherited by their $N\to\infty$ limits, we have, in addition,
\begin{equation} 
\begin{split} 
\lim_{K\to\infty}\,\lim_{N\to\infty}\,
\sum_{j=0}^k \widetilde\Sc^j_k(f_K,g_K;1)
=\lim_{K\to\infty}\,\int_0^\infty f_K(x)\,\big(\U^{\be;w}_Tg_K\big)(x)\d x \\
=\int_0^\infty f(x)\,\big(\U^{\be;w}_Tg\big)(x)\d x
\end{split}
\end{equation}
in distribution and in the sense of moments. 

\medskip

Moreover, the moment bounds of \eqref{Equation: Uniform Moment Bounds} reveal that 
\begin{equation}
\lim_{K\to\infty}\,\lim_{N\to\infty}\,\sum_{j=0}^k\bar\Sc^j_k(f,g,K;1)=0
\end{equation}
in $L^n$, for all $n\in\N$. To analyze the third sum on the right-hand side of \eqref{Equation: Weak Convergence Prelimit Final} we introduce,  for all $j=1,\,2,\,\ldots$, the notations
\begin{equation}
h_j(N)=\sum_{0\leq i_1\leq\cdots\leq i_j\leq k}\,\prod_{j'=1}^j\a_{\hat X^{k-j;x}_{N^{-2/3}i_{j'}}},\quad
p_j(N)=\sum_{i=0}^{k-j-H(X^{k-j;x})} \big(\a_{\hat X^{k-j;x}_{N^{-2/3}i}}\big)^j,
\end{equation}
and $[z^j]P(z)$ for the coefficient of $z^j$ in a power series $P(z)$. Then, the Newton identities relating the complete homogeneous symmetric functions to the power sums (see e.g. \cite[Chapter 1, Section 2]{M}) yield 
\begin{equation}
\frac{h_j(N)}{(2\sqrt{N})^j}=\frac{p_1(N)^j}{j!\,(2\sqrt{N})^j}
+\sum_{\iota=0}^{j-1} \frac{p_1(N)^\iota}{\iota!\,(2\sqrt{N})^\iota}\Bigg([z^{j-\iota}]\exp\Bigg(\sum_{j'=2}^\infty \frac{p_{j'}(N)}{j'(2\sqrt{N})^{j'}} z^{j'}\Bigg)\Bigg).
\end{equation}
Therefore, with
\begin{equation}
Z_N^{j,\iota}:=[z^{j-\iota}]\exp\Bigg(\sum_{j'=2}^\infty \frac{p_{j'}(N)}{j'(2\sqrt{N})^{j'}} z^{j'}\Bigg),\quad \iota=0,\,1,\,\ldots,\,j-1,
\end{equation}
it holds 
\begin{equation}
\sum_{j=0}^k \Delta^j_k(f,g)=\sum_{j=0}^k \,\sum_{\iota=0}^{j-1} 
\bar\Sc^\iota_k(f,g;Z_N^{j,\iota}).
\end{equation}
By \cite[Lemma 4.20]{GS}, one can find bounds $\E\big[|Z_N^{j,\iota}|^n\big]\le\Theta(3n,j-\iota,N)$ such that 
\begin{equation}\label{eq:property constants}
\lim_{N\to\infty}\,\sum_{j=0}^k\,\sum_{\iota=0}^{j-1}\frac{\Theta(3n,j-\iota,N)^{1/(3n)}}{(3/2)^j}=0.
\end{equation}
A combination of the triangle inequality for the $L^n$ norm, the moment bounds of \eqref{Equation: Uniform Moment Bounds} and the property \eqref{eq:property constants} gives
\begin{equation}
\lim_{N\to\infty}\,\sum_{j=0}^k \Delta^j_k(f,g)=0
\end{equation}
in $L^n$, for all $n\in\N$. This finishes the proof of Theorem \ref{Theorem: Main} under the continuity assumption on the $g$'s.

\medskip

For general $f,g\in{\mathcal D}$, the same arguments as above reveal that it suffices to identify, for all $K\in\N$, the limit of $\sum_{j=0}^k\widetilde\Sc^j_k(f_K,g_K;1)$ in distribution and in the sense of moments as $\int_0^\infty f_K(x)\,\big(\U^{\be;w}_Tg_K\big)(x)\d x$. In fact, it is enough to establish 
\begin{equation}\label{eq:approx_arg1}
\lim_{N\to\infty}\,\sum_{j=0}^k\widetilde\Sc^j_k(f_K,g_K;1)=\int_0^\infty f_K(x)\,\big(\U^{\be;w}_Tg_K\big)(x)\d x
\end{equation}
in distribution, since the moment bounds of \eqref{Equation: Uniform Moment Bounds} for $\bar\Sc^j_k(f_K,g_K,0;1)$ then imply the convergence of moments. To see \eqref{eq:approx_arg1} we pick $g_{\eta,K}$, $\eta\in\N$ in $C([0,\infty))$ so that
\begin{equation}
\upsilon_\eta:=\|g_{\eta,K}h_K-g_K\|_{L^2([0,\infty))}\underset{\eta\to\infty}{\longrightarrow}0. 
\end{equation}
Recalling the symmetry of $\widetilde\Sc^j_k(\,\cdot\,,\,\cdot\,;1)$ (cf. \eqref{eq:symmetry}), applying the Cauchy-Schwarz inequality, and repeating the proof of Proposition \ref{Proposition: Uniform Moment Bounds} mutatis mutandis we get 
\begin{equation}
\begin{split}
& \E\Big[\big|\widetilde\Sc^j_k(f_K,g_K;1)-\widetilde\Sc^j_k(f_K,g_{\eta,K}h_K;1)\big|^2\Big] \\
& \le \upsilon_\eta^2\;\E\bigg[\int_0^{N^{2/3}} \E_{X^{k-j;x}}\bigg[\widetilde F_j(X^{k-j;x},\a,\xi)\cdot N^{1/3}\int_{N^{-1/3}X^{k-j;x}_{T_{k-j}}}^{N^{-1/3}(X^{k-j;x}_{T_{k-j}}+1)}f_K(y)\d y\bigg]^2\d x\bigg] \\
&\le \upsilon_\eta^2\;\frac{Ce^{O(jN^{-\eps})}}{2^{2j}}.
\end{split}
\end{equation}
To complete the proof of Theorem \ref{Theorem: Main} we observe that
\begin{equation}
\forall\,\eta\in\N:\;\lim_{N\to\infty}\,\sum_{j=0}^k\widetilde\Sc^j_k(f_K,g_{\eta,K}h_K;1)=\int_0^\infty f_K(x)\,\big(\U^{\be;w}_T(g_{\eta,K}h_K)\big)(x)\d x
\end{equation}
in distribution and that
\begin{equation}
\lim_{\eta\to\infty} \int_0^\infty f_K(x)\,\big(\U^{\be;w}_T(g_{\eta,K}h_K)\big)(x)\d x = \int_0^\infty f_K(x)\,\big(\U^{\be;w}_Tg_K\big)(x)\d x
\end{equation} 
almost surely, thanks to Proposition \ref{Proposition: Properties}(c). 

\section{Properties of the limiting operators} \label{sec:properties}

The goal of this section is to prove Proposition \ref{Proposition: Properties}. We start by preparing some auxiliary constructions and results. From this point on, the entries of $M^{\be;w}_N$ are specified as in Remark \ref{rmk:special entries}. We also let  
\begin{equation}
A^{\be;w}_N:=N^{1/6}\big(2\sqrt{N}-M^{\be;w}_N\big),
\end{equation}
viewed as an operator acting on $L^2([0,\infty))$ via Remark \ref{rmk:matrix as operator}. The next proposition is a direct corollary of \cite[Proposition 2.8, Remark 2.9, Lemma 2.7, Theorem 2.10 and its proof]{BV1} (note that \cite[Assumptions 1-3]{BV1} for our particular model can be verified as in \cite[Section 3, last paragraph]{BV1}).

\begin{proposition} \label{Theorem: Detailed Spiked SAO is Edge Limit}
For all $\be>0$ and $w\in\R$, the operator $\H^{\be;w}$ of \eqref{eq:SSAE} almost surely possesses a purely discrete spectrum $\Lambda_1<\Lambda_2<\cdots$ satisfying $\La_q\to\infty$ as $q\to\infty$. The corresponding eigenspaces are one-dimensional, and each is therefore spanned by a normalized eigenfunction $f_q$. Moreover, one can couple $\H^{\be;w}$ with a subsequence of $A^{\be;w}_N$, $N\in\N$ along which, almost surely,
\begin{equation}
\lim_{N\to\infty} \La_{q,N}=\La_{q}\;\;\text{and}\;\;
\lim_{N\to\infty} \|v_{q,N}-f_{q}\|_{L^2([0,\infty))}=0,\quad q=1,\,2,\,\ldots,
\end{equation}
where $\La_{1,N}<\La_{2,N}<\cdots<\La_{N,N}$ and $v_{1,N},\,v_{2,N},\,\ldots,\,v_{N,N}$ are eigenvalues and corresponding eigenfunctions of $A^{\be;w}_N$. Along the same subsequence, the standard Brownian motion $W$ in the definition of $\H^{\be;w}$ arises in the almost sure limit \eqref{Equation: Origin of Noise}.
\end{proposition} 

Next, we present an alternative formula for the kernel $K^{\be;w}_T$ of \eqref{Equation: Spiked Kernel}, to be used in the proof of Proposition \ref{Proposition: Properties}.

\begin{lemma}\label{Lemma: Kernel Equivalences}
For all $\be>0$, $w\in\R$ and $T>0$, define the kernels
\begin{equation}
\begin{split}
& K^{\be}_T(x,y)=
\frac{\exp\left(-\frac{(x-y)^2}{2T}\right)}{\sqrt{2\pi T}} \\
& \quad\;\;\; \cdot\E_{\widetilde{W}^x}\bigg[\1_{\{\min_{0\le t\le T} \widetilde{W}^x_t>0\}}
\exp\bigg(-\int_0^T\frac{\widetilde{W}^x_t}2\d t+\int_0^\infty\frac{L_T^a(\widetilde{W}^x)}{\sqrt{\be}}\d W_a\bigg)\bigg|\widetilde{W}^x_T=y\bigg]
\end{split}
\end{equation}
and
\begin{equation}
\begin{split}
& \bar K^{\be;w}_T(x,y)=
\frac{2\exp\left(-\frac{(x-y)^2}{2T}\right)}{\sqrt{2\pi T}}\;\E_{\widetilde{W}^x}\bigg[\1_{\{\min_{0\le t\le T} \widetilde{W}^x_t\le 0\}} \\
& \qquad\qquad\;\cdot\exp\bigg(-\int_0^T\frac{|\widetilde{W}^x_t|}2\d t+\int_0^\infty\frac{L_T^a(|\widetilde{W}^x|)}{\sqrt{\be}}\d W_a-w\frac{L_T^0(|\widetilde{W}^x|)}2\bigg)\bigg|\widetilde{W}^x_T=y\bigg],
\end{split}
\end{equation} 
where $\widetilde{W}^x$ is a Brownian motion started at $x$, independent of $W$. Then,
\begin{equation}\label{eq:kernel decomposition}
K^{\be;w}_T(x,y)=K^{\be}_T(x,y)+\bar K^{\be;w}_T(x,y),\quad x,y\ge0.
\end{equation}
\end{lemma}

\begin{proof}
With $R^x:=|\widetilde{W}^x|$, the formula for $K^{\be;w}_T$ in  \eqref{Equation: Spiked Kernel} can be rewritten as  
\begin{equation}\label{eq:two expectations}
\begin{split}
& \frac{\exp\left(-\frac{(x-y)^2}{2T}\right)}{\sqrt{2\pi T}} \\
& \;\;\;\cdot\E_{\widetilde{W}^x}\bigg[\exp\bigg(-\int_0^T\frac{|\widetilde{W}^x_t|}2\d t+\int_0^\infty\frac{L_T^a(|\widetilde{W}^x|)}{\sqrt{\be}}\d W_a-w\frac{L_T^0(|\widetilde{W}^x|)}2\bigg)\bigg|\widetilde{W}^x_T=y\bigg] \\
& + \frac{\exp\left(-\frac{(x+y)^2}{2T}\right)}{\sqrt{2\pi T}} \\
& \;\;\;\cdot\E_{\widetilde{W}^x}\bigg[\exp\bigg(-\int_0^T\frac{|\widetilde{W}^x_t|}2\d t+\int_0^\infty\frac{L_T^a(|\widetilde{W}^x|)}{\sqrt{\be}}\d W_a-w\frac{L_T^0(|\widetilde{W}^x|)}2\bigg)\bigg|\widetilde{W}^x_T=-y\bigg].
\end{split}
\end{equation}
Next, we decompose the first expectation in \eqref{eq:two expectations} according to the events 
\begin{equation}
\Big\{\min_{0\le t\le T} \widetilde{W}^x_t>0\Big\}\quad\text{and}\quad
\Big\{\min_{0\le t\le T} \widetilde{W}^x_t\le 0\Big\},
\end{equation}
and note that, on the former event, $|\widetilde{W}^x|=\widetilde{W}^x$ and $L^0_T(|\widetilde{W}^x|)=0$.  

\medskip

In addition, by the strong Markov property and the symmetry about $0$ of Brownian motion, instead of conditioning on $\widetilde{W}^x_T=-y$ in the second expectation of \eqref{eq:two expectations}, we can condition on $\min_{0\le t\le T} \widetilde{W}^x_t\le 0$, $\widetilde W^x_T=y$. This operation, in turn, is equivalent to inserting $\mathbf{1}_{\{\min_{0\le t\le T} \widetilde{W}^x_t\le 0\}}$ into the expectation, conditioning on $\widetilde W^x_T=y$ and normalizing the result by $\Pr\big[\min_{0\le t\le T} \widetilde{W}^x_t\le 0\big|\widetilde W^x_T=y\big]$. Computing  
\begin{equation}\label{eq:BB reaches 0}
\Pr\Big[\min_{0\le t\le T} \widetilde{W}^x_t\le 0\Big|\widetilde W^x_T=y\Big]=e^{-2xy/T} 
\end{equation}
from the joint density of the running minimum and the current value of a Brownian motion (see e.g. \cite[Chapter III, Exercise 3.14]{RY}) and observing
\begin{equation}
e^{2xy/T}\cdot\frac{\exp\left(-\frac{(x+y)^2}{2T}\right)}{\sqrt{2\pi T}}
=\frac{\exp\left(-\frac{(x-y)^2}{2T}\right)}{\sqrt{2\pi T}}
\end{equation}
we arrive at the right-hand side of \eqref{eq:kernel decomposition}. 
\end{proof}

We are now ready to prove Proposition \ref{Proposition: Properties}.

\medskip

\noindent\textbf{Proof of Proposition \ref{Proposition: Properties}. (a).} The identity \eqref{eq:relation to SAO} follows from Theorem \ref{Theorem: Main} by the same arguments as were used to obtain \cite[Corollary 2.12]{GS} from Theorem 2.8 therein. One only needs to replace every reference to the main result of \cite{RRV} by a reference to Proposition \ref{Theorem: Detailed Spiked SAO is Edge Limit}, the pointer to \cite[Lemma 6.1]{GS} by a pointer to \eqref{eq:N2/3heuristic}, and the assertion that the eigenvalues of $-\frac 12\H^\be$ tend to $-\infty$ by the same statement for $-\frac 12\H^{\be;w}$ (cf. Proposition \ref{Theorem: Detailed Spiked SAO is Edge Limit}).

\medskip

\noindent\textbf{(b), (c).} We proceed to the almost sure Hilbert-Schmidt property of $\U^{\be;w}_T$, for each $T>0$. In view of Lemma \ref{Lemma: Kernel Equivalences},
it is enough to show that
\begin{equation}
\E\bigg[\int_0^\infty\int_0^\infty \big(K^{\be}_T(x,y)+\bar K^{\be;w}_T(x,y)\big)^2\d x \d y\bigg]<\infty.
\end{equation}
Since 
\begin{equation}
\E\bigg[\int_0^\infty\int_0^\infty K^{\be}_T(x,y)^2\d x \d y\bigg]<\infty
\end{equation}
is established in \cite[proof of Lemma 5.1]{GS}, it suffices to check 
\begin{equation}\label{Equation: Hilbert-Schmidt Expectation}
\int_0^\infty\int_0^\infty \E\big[\bar K^{\be;w}_T(x,y)^2\big]\d x\d y<\infty.
\end{equation}

\smallskip

Next, we estimate $\E\big[\bar K^{\be;w}_T(x,y)^2\big]$ by moving the square function into the expectation $\E_{\widetilde{W}^x}[\,\cdot\,|\widetilde{W}^x_T=y]$, dropping the indicator random variable, employing
\begin{equation}\label{eq:loc time abs}
L^a_T(|\widetilde{W}^x|)^2=\big(L^a_T(\widetilde{W}^x)
+L^{-a}_T(\widetilde{W}^x)\big)^2
\le 2L^a_T(\widetilde{W}^x)^2+2L^{-a}_T(\widetilde{W}^x)^2,\quad a\ge 0,
\end{equation}
and evaluating the expectation with respect to $W$:  
\begin{equation}
\begin{split}
& \E\left[\bar K^{\be;w}_T(x,y)^2\right]
\leq \frac{2\exp\left(-\frac{(x-y)^2}{T}\right)}{\pi T} \\
& \qquad\quad\;\;\cdot\E\bigg[\exp\bigg(-\int_0^T |\widetilde{W}^x_t|\d t+\int_{-\infty}^\infty\frac{4L_T^a(\widetilde{W}^x)^2}{\be}\d a-2wL_T^0(\widetilde{W}^x)\bigg)\bigg|\widetilde{W}^x_T=y\bigg].
\end{split}
\end{equation}
According to H\"older's inequality, the latter expectation is at most
\begin{equation}\label{eq:Hol factors}
\begin{split}
\E\bigg[\exp\bigg(\!\!-\int_0^T\!\! 3|\widetilde{W}^x_t|\d t\bigg)\bigg|\widetilde{W}^x_T=y\bigg]^{1/3}
\E\bigg[\exp\bigg(\int_{-\infty}^\infty\frac{12L_T^a(\widetilde{W}^x)^2}{\be}\d a\bigg)\bigg|\widetilde{W}^x_T=y\bigg]^{1/3} \\
\cdot\E\Big[e^{-6wL_T^0(\widetilde{W}^x)}\Big|\widetilde{W}^x_T=y\Big]^{1/3}.
\end{split}
\end{equation}

\smallskip

Thanks to $|\widetilde{W}^x_t|\ge \widetilde{W}^x_t$, the identity in distribution
\begin{equation}
\Big(\widetilde{W}^x_t:\,t\in[0,T]\,\Big|\,\widetilde{W}^x_T=y\Big)
\deq\Big(\widetilde{W}^0_t+\big(1-\tfrac tT\big)x+\tfrac tTy:\,t\in[0,T]\,\Big|\,\widetilde{W}^0_T=0\Big),
\end{equation}
and $\big(1-\frac tT\big)x+\frac tTy\geq x\land y$, the first factor in \eqref{eq:Hol factors} is bounded above by
\begin{equation}
e^{-T(x\land y)}\,\E\bigg[\exp\bigg(-\int_0^T 3\widetilde{W}^0_t \d t\bigg)\bigg|\widetilde{W}^0_T=0\bigg]^{1/3}.
\end{equation}
In addition, the coupling of \cite[Proposition 4.1]{GS} reveals the random variable $\int_{-\infty}^\infty L^a_T(\widetilde{W}^x)^2\d a$, conditioned on $\widetilde{W}^x_T=y$, as the almost sure $N\to\infty$ limit of the left-hand side in \cite[inequality (4.15)]{GS}. Thus, the second expectation in \eqref{eq:Hol factors} can be controlled by the limit inferior of the corresponding exponential moment of the right-hand side in \cite[inequality (4.15)]{GS}. Proceeding as therein we arrive at
\begin{equation}
\E\bigg[\exp\bigg(\int_{-\infty}^\infty\frac{12L_T^a(\widetilde{W}^x)^2}{\be}\d a\bigg)\bigg|\widetilde{W}^x_T=y\bigg]^{1/3}
\leq Ce^{C|x-y|},
\end{equation}
with a constant $C=C(\be,T)<\infty$. Also, we see from \cite[equation (3)]{Pitman} that the density of the local time at $0$ of a Brownian bridge from $x'$ to $y'$ on $[0,1]$ is
\begin{equation}\label{eq: loc time dens}
(z+x'+y')\exp\bigg(\frac12\Big((x'-y')^2-(z+x'+y')^2\Big)\bigg),\quad z>0
\end{equation}
and from \eqref{eq:BB reaches 0} that this local time vanishes with probability $1-e^{-2x'y'}$. Hence,
\begin{equation}
\begin{split}
& \E\big[\exp\big(\theta L^0_1(\widetilde{W}^{x'})\big)
\big|\widetilde{W}^{x'}_1=y'\big] \\
& =1+\sqrt{\frac{\pi}{2}}\theta e^{-2x'y'}\exp\bigg(\frac{(x'+y'-\theta)^2}{2}\bigg)\bigg(2-\erfc\bigg(\frac{-x'-y'+\theta}{\sqrt{2}}\bigg)\bigg)\le C(\theta)<\infty
\end{split}
\end{equation}
due to standard estimates for the complementary error function. All in all, it follows that the left-hand side in \eqref{Equation: Hilbert-Schmidt Expectation} is less or equal to 
\begin{equation}
\int_0^\infty\int_0^\infty C\exp\bigg(-\frac{(x-y)^2}{C}-\frac{x\land y}{C}+C|x-y|\bigg)\d x \d y<\infty,
\end{equation}
where $C=C(\beta,w,T)$ is a finite positive constant. 

\medskip

We turn to the proof of the semigroup property in Proposition \ref{Proposition: Properties}(b) and assume without loss of generality $T_1,T_2>0$. By the just established Hilbert-Schmidt property, it suffices to verify that, almost surely, 
\begin{equation}
\U^{\be;w}_{T_1}\U^{\be;w}_{T_2}f=\U^{\be;w}_{T_1+T_2}f
\end{equation}
on a countable dense set of functions $f\in L^2([0,\infty))$. Fixing such a function $f$, writing it as the difference of its positive and negative parts, and applying Fubini's theorem we reduce the statement of Proposition \ref{Proposition: Properties}(b) further to
\begin{equation}
\int_0^\infty K^{\be;w}_{T_1}(x,z) K^{\be;w}_{T_2}(z,y)\d z
=K^{\be;w}_{T_1+T_2}(x,y),\quad x,y\geq0.
\end{equation}

\smallskip

Let us introduce the transition kernels
\begin{equation}
\gamma_T(x,y):=\frac{\exp\left(-\frac{(x-y)^2}{2T}\right)+\exp\left(-\frac{(x+y)^2}{2T}\right)}{\sqrt{2\pi T}},\;\;x,y\geq0,\quad T>0
\end{equation}
and the additive functionals
\begin{equation}
F_T(R^x):=-\int_0^T\frac{R^x_t}2\d t+\int_0^\infty\frac{L_T^a(R^x)}{\sqrt{\be}}\d W_a-w\frac{L_T^0(R^x)}2,\quad T>0.
\end{equation}
Then,
\begin{equation}\label{Equation: Semigroup Integration}
\begin{split}
& \int_0^\infty K^{\be;w}_{T_1}(x,z) K^{\be;w}_{T_2}(z,y)\d z =\gamma_{T_1+T_2}(x,y) \\
& \qquad\quad\;\; \cdot\int_0^{\infty} \frac{\gamma_{T_1}(x,z)\gamma_{T_2}(z,y)}{\gamma_{T_1+T_2}(x,y)}
\E_{R^x}\Big[e^{F_{T_1}(R^x)}\Big|R^x_{T_1}=z\Big]
\E_{R^z}\Big[e^{F_{T_2}(R^z)}\Big|R^z_{T_2}=y\Big]\d z.
\end{split}
\end{equation}
To identify the right-hand side of \eqref{Equation: Semigroup Integration} with $K^{\be;w}_{T_1+T_2}(x,y)$ it remains to notice that the process $(R^x_t:\,t\in[0,T_1+T_2]\,|\,R^x_{T_1+T_2}=y)$ therein can be sampled by
\begin{enumerate}[(a)]
\item picking a random point $Z$ according to the density 
$\frac{\gamma_{T_1}(x,z)\gamma_{T_2}(z,y)}{\gamma_{T_1+T_2}(x,y)}$, $z>0$,
\item conditional $Z=z$, sampling processes $R^{(1)}$, $R^{(2)}$ independently such that
\begin{eqnarray}
&& (R^{(1)}_t:\,t\in[0,T_1])\deq(R^x_t:\,t\in[0,T_1]\,|\,R^x_{T_1}=z),\\
&& (R^{(2)}_t:\,t\in[0,T_2])\deq(R^z_t:\,t\in[0,T_2]\,|\,R^z_{T_1}=y),
\end{eqnarray}
\item concatenating the paths of $R^{(1)}$ and $R^{(2)}$.  
\end{enumerate}

\smallskip

As with the semigroup property, for each $T>0$, the symmetry property of the operator $\U_T$ can be reduced to an assertion about its kernel:
\begin{equation}\label{eq:symm kernel}
K^{\be;w}_T(x,y)=K^{\be;w}_T(y,x),\quad x,y\geq0.
\end{equation}
Since the transition kernels $\gamma_{T'}$, $T'>0$ of the reflected Brownian motion $R$ are symmetric, we have 
\begin{equation}
\big(R^x_t:\,t\in[0,T]\,\big|\,R^x_T=y\big)
\deq\big(R^y_{T-t}:\,t\in[0,T]\,\big|\,R^y_T=x\big),
\end{equation}
and therefore \eqref{eq:symm kernel}. Finally, the non-negativity of $\U^{\be;w}_T$ follows by extending 
\begin{equation}
\int_0^\infty \!\! (\U^{\be;w}_Tf)(x)\,f(x)\d x
\!=\!\int_0^\infty \!\! \big(\U^{\be;w}_{T/2}(\U^{\be;w}_{T/2}f)\big)(x)\,f(x) \d x
\!=\!\int_0^\infty \!\! \big(\!(\U^{w;\be}_{T/2}f)(x)\big)^2 \d x
\ge 0,
\end{equation}
for a fixed $f\in L^2([0,\infty))$, to the same almost sure property for all $f\in L^2([0,\infty))$ simultaneously, by means of the almost sure Hilbert-Schmidt property of $\U^{\be;w}_T$.

\medskip

\noindent\textbf{(d).} To obtain the almost sure trace class property of $\U^{\be;w}_T$ and the trace formula \eqref{Equation: Spiked Trace Formula} we combine the spectral theorem for symmetric compact operators with the definition of the trace to find
\begin{equation}
\Tr\big(\U^{\be;w}_T\big)=\sum_{q=1}^\infty e^{-T\Lambda_q/2}. 
\end{equation}
The latter sum is the square of the Hilbert-Schmidt norm of the symmetric Hilbert-Schmidt operator $\U^{\be;w}_{T/2}$ (see e.g. \cite[Section 28, Exercise 11]{Lax}) and, thus, equals to
\begin{equation}
\int_0^\infty \int_0^\infty K^{\beta;w}_{T/2}(x,y)\,K^{\beta;w}_{T/2}(y,x) \d y \d x
=\int_0^\infty K^{\beta;w}_T(x,x) \d x.
\end{equation}

\smallskip

\noindent\textbf{(e).} For the $L^2$-strong continuity in expectation of \eqref{Equation: L2 strong}, without loss of generality we fix an integer $p\ge 2$, an $f$ with $\|f\|_{L^2([0,\infty))}=1$, and a sequence $(t_\eta)_{\eta\in\N}$ in $[0,T+1]$ converging to $T$ such that $t_\eta>0$ for at least one $\eta$. Then, by applying \cite[Section 28, Theorem 7]{Lax} to the commuting symmetric operators $\U^{\be;w}_{t_\eta}$, $\eta\in\N$ and $\U^{\be;w}_T$, with at least one $\U^{\be;w}_{t_\eta}$ being compact, we can write $f$ as $\sum_{q=1}^\infty c_q f_q$, where $f_q$, $q\in\N$ form an orthonormal basis of common eigenfunctions for $\U^{\be;w}_{t_\eta}$, $\eta\in\N$ and $\U^{\be;w}_T$, and $c_q$, $q\in\N$ are the corresponding coefficients. By Jensen's inequality,
\begin{equation}
\begin{split}
\E\Big[\big\|\U^{\be;w}_Tf-\U^{\be;w}_{t_\eta} f\big\|^p_{L^2([0,\infty))}\Big]
& = \E\bigg[\Big(\sum_{q=1}^\infty c_q^2 \big(e^{-T\Lambda_q/2}-e^{-t_\eta\Lambda_q/2}\big)^2\Big)^{p/2}\bigg] \\
& \le \E\bigg[\sum_{q=1}^\infty c_q^2 \big(e^{-T\Lambda_q/2}-e^{-t_\eta\Lambda_q/2}\big)^p\bigg].
\end{split}
\end{equation}
The random variable $(\omega,q)\mapsto \big(e^{-T\Lambda_q(\omega)/2}-e^{-t_\eta\Lambda_q(\omega)/2}\big)^p$ tends to $0$ in the $\eta\to\infty$ limit $\Pr\times\sum_{q=1}^\infty \delta_{c_q^2}$ almost surely. Its uniform integrability is due to 
\begin{equation}
\E\bigg[\sum_{q=1}^\infty c_q^2 \big(e^{-T\Lambda_q/2}-e^{-t_\eta\Lambda_q/2}\big)^{2p}\bigg]
\le 2^{2p-1}\E\Big[2\big(e^{-p(T+1)\Lambda_1}+1\big)\Big]
\end{equation} 
and a bound on $e^{-p(T+1)\Lambda_1}$ by the squared Hilbert-Schmidt norm of $\U^{\be;w}_{p(T+1)}$, whose expectation has been controlled in the proof of part (c). \ep

\section{Functionals of the reflected Brownian bridge}\label{sec:Hariya}

In this section, we prove Theorem \ref{Theorem: Reflected Bridge-Local Time is Gaussian}, Proposition \ref{Proposition: Reflected Bridge Formula} and Corollary \ref{Corollary: Expectation of Spiked Kernel at Zero}, in the order stated. The key ingredient in the proof of Theorem \ref{Theorem: Reflected Bridge-Local Time is Gaussian} is the next lemma, which extends an argument of Hariya \cite{H}.

\begin{lemma}\label{Lemma: Reduction} 
Let $r_t$, $t\in[0,1]$ be a reflected Brownian bridge, $\al>0$ and $\psi_\al$ be the joint moment-generating function of 
\begin{equation}
\bigg(\int_0^1r_{t}\d t,\,\int_0^\infty\frac{L^a_1(r)^2}2\d a\bigg|L^0_1(r)=\al\bigg).
\end{equation}
Then, with the three-dimensional Bessel bridge
\begin{equation}\label{def:Bes3br}
\d b^{\al/2}_t = \frac{1}{b^{\al/2}_t} \d t
-\frac{b^{\al/2}_t}{1-t} \d t + \dd \widetilde{W}_t,\quad b^{\al/2}_0=\al/2
\end{equation}
and the joint moment-generating function $\widetilde{\psi}_\al$ of $\big(\int_0^1 b^{\al/2}_{t}\d t,\,\int_0^1\widetilde{W}_t\d t\big)$, it holds 
\begin{equation}\label{eq: mom gen iden}
\psi_\al(\theta_1,\theta_2)=e^{-\al\theta_1/4}\,
\widetilde{\psi}_\al(\theta_1+2\theta_2,-\theta_1/2),\quad \theta_1,\theta_2\in\R.
\end{equation}
\end{lemma}

\begin{proof}
Define the function
\begin{equation}
h(a):=\int_0^1 \1_{\{r_t\leq a\}}\d t=\int_0^a L^{a'}_1(r)\d a',
\end{equation}
as well as the corresponding quantile function $h^{-1}(t):=\inf\{a\geq0: h(a)\geq t\}$. In \cite[Corollary 16(iii)]{Pitman2} (see also \cite[equation (8.20)]{Pitman0}), Pitman shows that 
\begin{equation}\label{Equation: Generalized Jeulin}
\bigg(\frac12 L^{h^{-1}(t)}_1(r):\,t\in[0,1]\Big|L^0_1(r)=\al\bigg)
\deq\big(b^{\al/2}_{t}:\,t\in[0,1]\big),
\end{equation}
which extends Jeulin's theorem beyond the $\al=0$ case. Relying on \eqref{Equation: Generalized Jeulin} we find
\begin{equation}\label{eq: Jeulin man 1}
\begin{split}
&\bigg(\frac12\int_0^1\frac{1-t}{b^{\al/2}_{t}}\d t,\,\int_0^1b^{\al/2}_{t}\d t\bigg) \\
&\deq\bigg(\int_0^1\frac{1-t}{L^{h^{-1}(t)}_1(r)}\d t,\,\frac12\int_0^1L^{h^{-1}(t)}_1(r)\d t\bigg| L^0_1(r)=\al\bigg) \\
&=\bigg(\int_0^\infty \frac{1-h(a)}{L^a_1(r)} h'(a)\d a,\,\frac12\int_0^\infty L^a_1(r) h'(a)\d a\bigg| L^0_1(r)=\al\bigg) \\
&=\bigg(\int_0^\infty 1-h(a)\d a,\,\frac12\int_0^\infty L^a_1(r)^2\d a\bigg| L^0_1(r)=\al\bigg) \\
&=\bigg(\int_0^\infty\int_0^1\1_{\{r_t>a\}}\d t\d a,\,\frac12\int_0^\infty L^a_1(r)^2\d a\bigg| L^0_1(r)=\al\bigg) \\
&=\bigg(\int_0^1r_t\d t,\,\frac12\int_0^\infty L^a_1(r)^2\d a\bigg| L^0_1(r)=\al\bigg).
\end{split}
\end{equation}

\smallskip

On the other hand, \eqref{def:Bes3br} implies
\begin{equation}
\int_0^1(1-t)\d b^{\al/2}_t
=\int_0^1 \frac{1-t}{b^{\al/2}_{t}} \d t - \int_0^1 b^{\al/2}_t \d t+\int_0^1 (1-t)\d \widetilde{W}_t.
\end{equation}
Using integration by parts for the two stochastic integrals and rearranging we get 
\begin{equation}\label{eq: Jeulin man 2}
\int_0^1\frac{1-t}{b^{\al/2}_{t}}\d t=-\frac{\al}2+2\int_0^1b^{\al/2}_{t}\d t-\int_0^1\widetilde W_t\d t.
\end{equation}
Finally, a sequential application of \eqref{eq: Jeulin man 1} and \eqref{eq: Jeulin man 2} yields
\begin{equation}
\begin{split}
&\;\E\bigg[\exp\bigg(\theta_1 \int_0^1 r_t\d t+\theta_2\int_0^\infty L^a_1(r)^2\d a\bigg)\bigg| L^0_1(r)=\al\bigg] \\
&=\E\bigg[\exp\bigg(\frac{\theta_1}2\int_0^1\frac{1-t}{b^{\al/2}_{t}}\d t+2\theta_2\int_0^1 b^{\al/2}_{t}\d t\bigg)\bigg] \\
&=e^{-\al\theta_1/4}\E\bigg[\exp\bigg((\theta_1+2\theta_2)\int_0^1 b^{\al/2}_t \d t-\frac{\theta_1}2\int_0^1\widetilde W_t\d t\bigg)\bigg],
\end{split}
\end{equation}
that is, \eqref{eq: mom gen iden}.
\end{proof}

Theorem \ref{Theorem: Reflected Bridge-Local Time is Gaussian} can be now obtained from Lemma \ref{Lemma: Reduction} as follows. 

\begin{proof}[Proof of Theorem \ref{Theorem: Reflected Bridge-Local Time is Gaussian}]
The case $\al=0$ is the subject of \cite[Corollary 2.15]{GS}, \cite[Theorem 1.1]{H}, so we focus on the $\al>0$ case. By Lemma \ref{Lemma: Reduction}, for every $\theta\in\R$,
\begin{equation}\label{eq:Laplace}
\E\bigg[\exp\bigg(\theta \int_0^1 r_t \d t-\frac\theta2\int_0^\infty L^a_1(r)^2\d a\bigg)\bigg|L^0_1(r)=\al\bigg] 
=e^{-\al\theta/4}\widetilde\psi_\al(0,-\theta/2).
\end{equation}
Since $\int_0^1 \widetilde W_t\d t$ is Gaussian with mean $0$ and variance $\frac13$, the right-hand side of \eqref{eq:Laplace} equals to $e^{-\alpha\theta/4+\theta^2/24}$, the moment-generating function of a Gaussian random variable with mean $-\al/4$ and variance $1/12$. 
\end{proof}

We conclude the paper with the proofs of Proposition \ref{Proposition: Reflected Bridge Formula} and Corollary \ref{Corollary: Expectation of Spiked Kernel at Zero}.

\begin{proof}[Proof of Proposition \ref{Proposition: Reflected Bridge Formula}]
Let $\widetilde{r}_t$, $t\in[0,T]$ be a reflected Brownian bridge from $0$ to $0$ on $[0,T]$. By the definition of $K^{\be;w}_T$ in \eqref{Equation: Spiked Kernel},
\begin{equation}
K^{\be;w}_T(0,0)=
\sqrt{\frac{2}{\pi T}}
\E_{\widetilde{r}}\bigg[\exp\bigg(-\int_0^T \frac{\widetilde{r}_t}2\d t
+\int_0^\infty\frac{L^a_T(\widetilde{r})}{\sqrt{\be}}\d W_a
-w\frac{L^0_{T}(\widetilde{r})}2\bigg)\bigg].
\end{equation}
Conditional on $\widetilde{r}$, the integral $\int_0^\infty\frac{L^a_T(\widetilde{r})}{\sqrt{\be}}\d W_a$ is Gaussian with mean $0$ and variance $\int_0^\infty \frac{L^a_T(\widetilde{r})^2}{\be}\d a$. Hence, by taking the expectation with respect to $W$ first, we find
\begin{equation}
\E\big[K^{\be;w}_T(0,0)\big]=\sqrt{\frac{2}{\pi T}}
\E\bigg[\exp\bigg(-\int_0^T \frac{\widetilde{r}_t}2\d t
+\int_0^\infty\frac{L^a_{T}(\widetilde{r})^2}{2\be}\d a
-w\frac{L^0_{T}(\widetilde{r})}2\bigg)\bigg].
\end{equation}
At this point, the proposition is a consequence of
\begin{equation}
\begin{split}
& \bigg(\int_0^T \widetilde{r}_t\d t,\,
\int_0^\infty L^a_T(\widetilde{r})^2\d a,\,
L^0_T(\widetilde{r})\bigg) \\
& \deq
\bigg(T^{3/2}\int_0^{1}r_t\d t,\,
T^{3/2}\int_0^\infty L^a_{1}(r)^2\d a,\,
T^{1/2}L^0_1(r)\bigg),
\end{split}
\end{equation}
which, in turn, is due to the scaling property of (reflected) Brownian bridges.
\end{proof}

\begin{proof}[Proof of Corollary \ref{Corollary: Expectation of Spiked Kernel at Zero}]
In view of \eqref{eq: loc time dens}, the local time $L^0_1(r)$ is a continuous random variable with the density $\frac \al 4 e^{-\al^2/8}$ on $(0,\infty)$. Using this and Theorem \ref{Theorem: Reflected Bridge-Local Time is Gaussian} for the right-hand side of \eqref{eq:exp_kernel} we compute
\begin{equation}
\begin{split}
\E\big[K^{2;w}_T(0,0)\big]
& =\sqrt{\frac{2}{\pi T}} \int_0^\infty \frac \al 4 e^{-\al^2/8}
\exp\Big(\!-T^{1/2}w\frac{\al}2\,\Big) \\
& \quad\cdot\E\bigg[\exp\bigg(-\frac{T^{3/2}}2\bigg(\int_0^Tr_t\d t-\frac12\int_0^\infty L^a_1(r)\d a\bigg)\bigg)\bigg|L^0_1(r)=\al\bigg]\d\al \\
&=\sqrt{\frac{2}{\pi T}} \int_0^\infty \frac \al 4 e^{-\al^2/8}
\exp\Big(\!-T^{1/2}w\frac{\al}2+\frac{T^{3/2}\al}8+\frac{T^3}{96}\,\Big)
\d\al \\
&=\frac{e^{\frac{T^3}{96}} \Big(8+\sqrt{2\pi}e^{T(T-4 w)^2/32}\sqrt{T}(T-4 w)\Big(\erf\Big(\frac{\sqrt{T}(T-4 w)}{4\sqrt{2}}\Big)+1\Big)\Big)}{4 \sqrt{2 \pi T}},
\end{split}
\end{equation}
which simplifies to the right-hand side of \eqref{eq: kernel at 0 expl}.
\end{proof}


%


\bigskip\bigskip\bigskip


\bigskip\bigskip\bigskip


\begin{thebibliography}{10}

\bibitem{AGZ}
G.~Anderson, A.~Guionnet, O.~Zeitouni (2010). \textit{An introduction to random matrices}. Cambridge Studies in Advanced Mathematics \textbf{118}. Cambridge University Press, Cambridge.  


\bibitem{BBP}
J.~Baik, G.~Ben Arous, S.~P\'ech\'e (2005).
Phase transition of the largest eigenvalue for nonnull complex sample covariance matrices. \textit{Ann. Probab.} \textbf{33}, pp. 1643--1697.

\bibitem{BK}
R.~Bass, D.~Khoshnevisan (1992). Strong approximations to Brownian local time. \textit{Progr. Probab.} \textbf{33}, pp. 43--65. 

\bibitem{BR}
R.~N.~Bhattacharya, R.~R.~Rao (1976). \textit{Normal approximation and asymptotic expansions}. John Wiley \& Sons, New York.

\bibitem{Bi} 
P.~Billingsley (1999). \textit{Convergence of probability measures}. 2nd ed. John Wiley \& Sons, New York.


\bibitem{BV1}
A.~Bloemendal, B.~Vir\'ag (2013). Limits of spiked random matrices I. \textit{Probab. Theory Related Fields} \textbf{156}, pp. 795--825.

\bibitem{BV2}
A.~Bloemendal, B.~Vir\'ag (2016).  Limits of spiked random matrices II.
\textit{Ann. Probab.} \textbf{44}, pp. 2726--2769. 



\bibitem{Ch} 
X.~Chen (2010). \textit{Random walk intersections: large deviations and related topics}. Mathematical Surveys and Monographs \textbf{157}. American Mathematical Society, Providence.



\bibitem{DE}
I.~Dumitriu, A.~Edelman (2002). Matrix models for beta ensembles.
\textit{J. Math. Phys.} \textbf{43}, pp. 5830--5847.

\bibitem{E}
A.~Edelman (2003). Stochastic differential equations and random matrices. \textit{SIAM conference on applied linear algebra}. The College of William and Mary, Williamsburg.

\bibitem{ES}
A.~Edelman, B.~Sutton (2007). From random matrices to stochastic operators. \textit{J. Stat. Phys.} \textbf{127}, pp. 1121--1165.


\bibitem{F}
P. Forrester (2010). \textit{Log-gases and random matrices}. London Mathematical Society Monographs Series \textbf{34}. Princeton University Press, Princeton. 

\bibitem{GS}
V.~Gorin, M.~Shkolnikov (2016). Stochastic Airy semigroup through tridiagonal matrices. \textit{arXiv:1601.06800v1}.

\bibitem{H}
Y.~Hariya (2016). A pathwise interpretation of the Gorin-Shkolnikov identity. \textit{Electron. Commun. Probab.} \textbf{21}, pp. 1--6.

\bibitem{HF}
D.~Holcomb, G.~R.~M.~Flores (2012). Edge scaling of the $\beta$-Jacobi ensemble. \textit{J. Stat. Phys.} \textbf{149}, pp. 1136--1160.


\bibitem{Ka}
R.~L.~Karandikar (1995). On pathwise stochastic integration. \textit{Stoch. Proc. Appl.} \textbf{57}, pp. 11--18. 

\bibitem{KRV}
M.~Krishnapur, B.~Rider, B.~Vir\'ag (2016). Universality of the stochastic Airy operator. \textit{Comm. Pure Appl. Math.} \textbf{69}, pp. 145--199.

\bibitem{Lax}
P.~Lax (2002). \textit{Functional analysis}. John Wiley \& Sons, New York.

\bibitem{LL}
G.~Lawler, V.~Limic (2010). \textit{Random walk: a modern introduction}. Cambridge Studies in Advanced Mathematics \textbf{123}.
Cambridge University Press, Cambridge.



\bibitem{M}
I.~G.~Macdonald (2015).
{\it Symmetric functions and Hall polynomials.}
2nd ed. The Clarendon Press, Oxford University Press, New York.



\bibitem{Pitman}
J.~Pitman (1999). The distribution of local times of a Brownian bridge. \textit{S\'eminaire de Probabilit\'es} \textbf{XXXIII}, pp. 388--394. 
Lecture Notes in Math. \textbf{1709}. Springer-Verlag, Berlin.

\bibitem{Pitman2}
J.~Pitman (1999). The SDE solved by local times of a Brownian excursion or bridge derived from the height profile of a random tree or forest.
\textit{Ann. Probab.} \textbf{27}, pp. 261--283. 

\bibitem{Pitman0}
J.~Pitman (2006). \textit{Combinatorial stochastic processes}. Lecture Notes in Math. \textbf{1875}. Springer-Verlag, Berlin.


\bibitem{RY}
D.~Revuz, M.~Yor (1999). \textit{Continuous martingales and Brownian motion.} 3rd ed. Grundlehren der Mathematischen Wissenschaften \textbf{293}. Springer-Verlag, Berlin.

\bibitem{RR}
J.~Ram\'\i rez, B.~Rider (2009). Diffusion at the random matrix hard edge. \textit{Comm. Math. Phys.} \textbf{288}, pp. 887--906.

\bibitem{RRV}
J.~Ram\'\i rez, B.~Rider, B.~Vir\'ag (2011). Beta ensembles, stochastic Airy spectrum, and a diffusion. \textit{J. Amer. Math. Soc.} \textbf{24}, pp. 919--944. 

\bibitem{RW}
B.~Rider, P.~Waters (2016). Universality of the stochastic Bessel operator. \textit{arXiv:1610.01637}.



\bibitem{St}
D.~Stroock (2011). \textit{Essentials of integration theory for analysis}. Graduate Texts in Math. \textbf{262}. Springer-Verlag, New York.

\bibitem{VV1}
B.~Valk\'o, B.~Vir\'ag (2009). Continuum limits of random matrices and the Brownian carousel. \textit{Invent. Math.} \textbf{177}, pp. 463--508.

\bibitem{VV2}
B.~Valk\'o, B.~Vir\'ag (2016). The Sine$_{\be}$ operator. \textit{arXiv:1604.04381}.

\end{thebibliography}
\end{document}